\documentclass[12pt,leqno]{article}

\usepackage{amsfonts,url}
\usepackage{enumerate}
\usepackage{color}
\usepackage[margin=30mm]{geometry}
\usepackage{amssymb, amsmath, amsthm, amscd}
\usepackage{eucal,mathrsfs,dsfont}

\def\dist{\mathop{\rm dist}\nolimits}
\def\supp{\mathop{\rm supp}\nolimits}

\newcommand{\R}{\mathbb{R}}
\newcommand{\Rd}{ \mathbb{R}^{d}}
\newcommand{\N}{\mathbb{N}}

\newcommand{\indyk}[1]{\mathds{1}_{#1}}
\newcommand{\sfera}{ \mathds{S}}

\newcommand{\Fourier}{ {\mathcal{F}}}

\newcommand{\nubounded}[1]{\bar{\nu}_{#1}}

\newcommand{\scalp}[2]{#1\cdot#2}

 \def\dist{\mathop{\rm
    dist}\nolimits} \def\diam{\mathop{\rm diam}\nolimits}

\newtheorem{lemat}{\indent\sc Lemma}[section]
\newtheorem{prop}[lemat]{\indent\sc Proposition}
\newtheorem{twierdzenie}[lemat]{\indent\sc Theorem}

\newcounter{conum} 

\renewcommand{\Re}{\ensuremath{\operatorname{Re}}}

\begin{document}

\title{Estimates of densities for L\'evy processes with lower intensity of large jumps}
\author{Pawe{\l} Sztonyk}
\footnotetext{ Pawe{\l} Sztonyk \\
  Faculty of Pure and Applied Mathematics,
  Wroc{\l}aw University of Technology,
  Wybrze{\.z}e Wyspia{\'n}\-skie\-go 27,
  50-370 Wroc{\l}aw, Poland.\\
  {\rm e-mail: Pawel.Sztonyk@pwr.edu.pl} \\
}\date{December 30, 2015}
\maketitle

\begin{center}
  Abstract
\end{center}
\begin{scriptsize}
  We obtain general lower estimates of transition densities of jump L\'evy processes. We use them for
	processes with L\'evy measures having bounded support, processes with exponentially 
	decaying L\'evy measures for large times and for processes with high intensity of small jumps for small times.
\end{scriptsize}

\footnotetext{2000 {\it MS Classification}:
Primary 60G51, 60E07; Secondary 60J35, 47D03, 60J45 .\\
{\it Key words and phrases}: L\'evy process, L\'evy measure, tempered stable process, semigroup of measures, transition density, heat kernel.\\
P. Sztonyk was supported by the National Science Center (Poland) grant on the basis of the decision No. DEC-2012/07/B/ST1/03356.
}
\section{Introduction}

Let $d\in\{1,2,\dots\}$ and $\nu$ be a symmetric L\'evy measure on $\Rd$, i.e.,
\begin{equation}\label{Levy1}
  \int_{\Rd} \left(1\wedge |y|^2\right)\,\nu(dy) < \infty,
\end{equation}
and $\nu(-D)=\nu(D)$ for every Borel set $D\subset\Rd$. We always assume also that $\nu(\Rd)=\infty$.

We consider the convolution semigroup of probability measures $\{ P_t,\, t\geq 0 \}$ with the Fourier transform $\Fourier(P_t)(\xi)=\int_{\Rd} e^{i\scalp{\xi}{y}}P_t(dy)=\exp(-t\Phi(\xi))$, where
\begin{displaymath}
  \Phi(\xi) =    \int_{\Rd} \left(1-\cos(\scalp{\xi}{y}) \right)\nu(dy) ,\quad \xi\in\Rd.
\end{displaymath}

There exists a L\'evy process $\{X_t,\,t\geq 0\}$
corresponding to $\{P_t,\,t\geq 0\}$, i.e., $P_t$ is the transition function of $X_t$.


We denote
$$
  \Psi(r)=\sup_{|\xi|\leq r}  \Phi(\xi),\quad r>0.
$$
It follows directly from the definition that $\Psi(|\xi|)\geq \Phi(\xi)$ for $\xi\in\Rd$. An opposite 
inequality $\Psi(|\xi|)\leq c \Phi(\xi)$ holds also in many typical examples but is not true in general.

We will often use the following estimate obtained in Proposition 1 in \cite{KSz2} (see also Lemma 6 in \cite{Grz2013})
\begin{equation}\label{eq:PsiH}
	 L_0 H(r) \leq \Psi(r) \leq 2 H(r),\quad r>0,
\end{equation}
where
\begin{displaymath}
	H(r) = \int \left( 1 \wedge r^2|y|^2 \right) \, \nu(dy),
\end{displaymath}
and $L_0$ depends only on the dimension $d$. 

We note that $\Psi$ is continuous and nondecreasing and $\sup_{r>0} \Psi(r)=\infty$, since $\nu(\Rd)=\infty$ 
(it follows easily from \eqref{eq:PsiH}). 
Let $\Psi^{-1}(s)=\sup\{r>0: \Psi(r)=s\}$ for
$s\in (0,\infty)$ so that $\Psi(\Psi^{-1}(s))=s$ and $\Psi^{-1}(\Psi(s))\geq s$ for $s>0$. Define
$$
  h(t)=\frac{1}{\Psi^{-1}\left(\frac{1}{t}\right)},\quad t>0.
$$

We often use the following condition which is satisfied under mild assumptions on $\Phi$ (see Lemma 5 in \cite{KSz2} or Lemma 5 in \cite{KSz3}). 

\[
  \begin{array}{c}
  \textit{There exist constants } M_0> 0, \textit{   and   } t_p \in (0,\infty] \textit{ such that} \\
  \int_{\Rd} e^{-t\Phi(\xi)}|\xi|\, d\xi \leq M_0 \left(h(t)\right)^{-d-1},
  \quad t\in (0,t_p).
  \end{array}
  \tag{{\bf A1}}\label{Assum}
\]

We note that \eqref{Assum} yields in particular that $\nu(\Rd)=\infty$ and the existence of the transition densities $p_t$ of $P_t$ for all $t>0$.

The main results of the present paper are the following two lower estimates of 
the transition densities. They contain universal minimal bounds
for jump L\'evy processes.

\begin{twierdzenie}\label{main1}
  For every symmetric L\'evy measure $\nu$ such that \eqref{Assum} 
	holds with $t_p=\infty$ there exists positive constants $c_1-c_4$, such that
\begin{displaymath}
	p_t(x)\geq c_1 h(t)^{-d} e^{\frac{-c_2|x|^2}{t}},\quad t>c_3,\, |x|\leq c_4 t,
\end{displaymath}
where $p_t$ is the density of $P_t$.
\end{twierdzenie}

Let $\nu = \nu_s + \nu_c$ where $\nu_s$ and $\nu_c$ are singular and continuous part of $\nu$
with respect to the Lebesgue measure on $\Rd\setminus\{ 0 \}$, respectively, and
let $\frac{d\nu_c}{dm}$ denote the Radon-Nikodym derivative of $\nu_c$.

\begin{twierdzenie}\label{main2}
  Assume that \eqref{Assum} holds and there exists $r_0>0$ such that
	\begin{displaymath}
	  \inf_{0<|y|<r_0} \frac{d\nu_c}{dm}(y) >0.	
	\end{displaymath}
Then there exist constants $c_1-c_4$, such that
\begin{displaymath}
  p_t(x) \geq c_1 e^{-c_2|x|\log\left(\frac{c_3 |x|}{t}\right)},
\end{displaymath}
for $t\in(0,t_p)$, and $|x|\geq\max\left\{ r_0, c_4 t \right\}$.
\end{twierdzenie}

We prove the theorems in Section \ref{section_lower}. Note that explicite values of the constants and also more specific estimates can be found in Lemma \ref{below_exp2a}, Lemma \ref{below_exp2b} and Lemma \ref{pbelow_expxlog} in Section \ref{section_lower}. We emphasise also that lower bounds obtained for 
transition densities in previous papers depend usually on the local behaviour of 
the L\'evy measure $\nu$ (see \cite{S11},\cite{W07}) or hold only for isotropic processes (\cite{BGR13},\cite{ChKum08}). In particular, 
Proposition \ref{th:p_est_below} below gives the lower bound in terms of the L\'evy measure:
$p_t(x) \geq c_1 t h(t)^{-d} \nu(B(x,c_2h(t)))$, for $|x|>c_3h(t)$ and
$t\in (0, t_p)$,
and the both above theorems deliver useful estimates 
even in regions on which $\nu$ is not supported.

Although the above estimates of transition densities hold for wide class of L\'evy processes
one can hardly expect that they are optimal for processes with heavy tails of the L\'evy
measure since it is known that $\nu(dx) = g(x) dx$ is a vague limit of measures $P_t(dx)/t = (p_t(x)/t) dx$
as $t\to 0^+$ outside the origin, and in fact the both
functions $p_t(x)$ and $tg(x)$ share typically the same asymptotic properties for such processes for small times
(see results and discussions in \cite{KSz2}, \cite{KSz3}). 
In particular for $\alpha$ - stable processes with $\nu(dx) \asymp |x|^{-d-\alpha}dx$
we have $p_t(x) \asymp t^{-d/\alpha}(1+t^{-1/\alpha}|x|)^{-d-\alpha}$.
Therefore
the above estimates are useful mainly for processes with truncated jumps or with exponentially decaying intensity of jumps and large times where the asymptotic of $p(t)$ and $tg(x)$ can differ significantly. In the next sections we give some applications and show that the above results
are optimal or close to optimal for the considered processes using existing upper estimates.

The first natural application are processes with truncated L\'evy measures. In \cite{ChKimKum2} 
the authors obtained both side estimates of transition densities for processes with truncated isotropic stable 
L\'evy measure. Here we extend the results of
\cite{ChKimKum2} to much wider class of processes with L\'evy measure with bounded support and not necessarily absolutely continuous. The lower estimates which follow easily from the above inequalities and 
Proposition \ref{th:p_est_below} are presented 
in Section \ref{app_truncated} in Theorem \ref{tw:zdolu_truncated}. Next we prove upper estimates in Lemma \ref{th:p_est_above} and Theorem \ref{tw:zgory_truncated}. Our method is based on the results of \cite{KSch2012} where the authors complement in very useful way 
the known results
of Carlen, Kusuoka, and Stroock (\cite{CKS}). Assuming additionally that the L\'evy measure
is absolutely continuous 
we obtain more precise following estimates which can be regarded as the third main result of the present paper. 
We note that similar estimate was announced (without a proof) in Theorem 1.4 of \cite{ChKimKum1}.

We will use here the following condition on a function $f:(0,r_0]\to (0,\infty)$.
\[
  \begin{array}{c}
  \textit{There exist constants } M_1,M_2 \geq  0 \textit{ and } d<\beta_1\leq \beta_2<d+2, \textit{ such that} \\
    M_1 \left(\frac{R}{r}\right)^{\beta_1} \leq \frac{f(r)}{f(R)} \leq M_2 \left( \frac{R}{r} \right)^{\beta_2},\quad r_0 \geq R\geq r >0.
  \end{array}
  \tag{{\bf A2}}\label{Assum2}
\]
We will use the notation $f\asymp g$ to indicate that there exist constants $c_1,c_2$ such that
$c_1 g \leq f\leq c_2 g$.

\begin{twierdzenie}\label{tw:obustronne}
   Assume that $\supp(\nu)\subset B(0,r_0)$, 
  $\nu$ is symmetric and absolutely continuous with respect to the Lebesgue measure on $\Rd\setminus\{0\}$ with
  a density $\bar\nu$ and there exists a nonincreasing function $f:(0,r_0]\to (0,\infty]$ 
  such that
  \begin{displaymath}
	  \bar\nu(x) \asymp f(|x|),\quad 0<|x|<r_0,
  \end{displaymath}
  where $f$ satisfies \eqref{Assum2} 
  and $\kappa = \inf_{s\in (0,r_0]} f(s) > 0$.
  
  Then $P_t$ is for every $t>0$ absolutely continuous with a density $p_t$ which satisfies the following estimates.
  \begin{enumerate}
    \item There exists $\eta_*$ such that for $|x| \leq \eta_* h(t)$ we have
      \begin{displaymath}
	      p_t(x) \asymp h(t)^{-d}.
      \end{displaymath}
    \item There exists $C^*$ such that for $\eta_* h(t) \leq |x| \leq r_0,\, t\leq t_1$ we have
      \begin{displaymath}
        p_t(x) \asymp t f(|x|),
      \end{displaymath}
      where $t_1 = r_0 / C^*$.
     \item There exist positive constants $c_1-c_{4}$  such that
      \begin{displaymath}
	      c_{1} h(t)^{-d} \exp\left\{-\frac{ c_{2} |x|^2}{ t}\right\} \leq p_t(x)
	      \leq c_{3} h(t)^{-d} \exp\left\{-\frac{ c_{4} |x|^2}{ t}\right\},
      \end{displaymath}
      for $\eta_* h(t) \leq |x| \leq C^* t, t\geq t_1$.
      \item There exist positive constants $c_5-c_{10}$ such that 
      \begin{displaymath}
	      c_5 \exp\left\{-c_6 |x| \log\left(\frac{c_7 |x|}{t}\right)\right\} \leq p_t(x) \leq 
	      c_8 \exp\left\{-c_9 |x| \log\left(\frac{c_{10} |x|}{t}\right)\right\},
      \end{displaymath}
      for $|x|\geq r_0 \vee C^* t,\, t>0$.
  \end{enumerate}
\end{twierdzenie}

The second interesting application are L\'evy processes with exponentially decaying L\'evy measure which we discuss in Section \ref{section_tempered}.
We extend here the results obtained previously
in \cite{ChKimKum1} and \cite{KSz3}.
The sharp estimates obtained in \cite{KSz3} hold only for small times whereas the results of \cite{ChKimKum1} contain
only absolutely continuous L\'evy measures. Here we obtain in Theorem \ref{t:esttemp1} both side estimates for large times and not
necessarily absolutely continuous L\'evy measures. 

\begin{twierdzenie}\label{t:esttemp1} 
  Let
  \begin{displaymath}
    \nu(A) \asymp \int_0^\infty \int_\sfera \indyk{A}(s\theta) s^{-1-\alpha} (1+s)^{\kappa} e^{-m s^\beta}\, 
		ds\mu(d\theta),
  \end{displaymath} 
  where $\mu$ is bounded, symmetric and nondegenerate measure on the unit sphere $\sfera$, 
  $m > 0$, $\beta\in (0,1]$, $\alpha\in (0,2)$,
  $\kappa\in (-\infty,1+\alpha]$. Then
  there exist constants $c_1-c_6,\eta,t_0$ such  that
  \begin{equation}\label{eq:temp1}
    p_t(x) \leq c_1 t^{-d/2} \left( e^{ \frac{-c_2|x|^2}{t}  } 
                + e^{\frac{-m|x|^\beta}{2\cdot 4^\beta}} \right),
  \end{equation}
  for $x\in\Rd,\, t> t_0,$
  and
  \begin{equation}\label{eq:temp2}
	  p_t(x) \geq c_3 t^{-d/2} \left( e^{ \frac{-c_4|x|^2}{t}  } 
                +  t \nu (B(x,c_5 \sqrt{t})) \right),\quad  \eta \sqrt{t} \leq |x| \leq c_6 t,\, t>t_0.
  \end{equation}
  In particular, if 
  \begin{equation}\label{eq:approxtemp}
    \nu(dx)\asymp |x|^{-d-\alpha} (1+|x|)^{\kappa} e^{-m |x|^\beta}\, dx,\quad x\in\Rd\setminus\{ 0\},
  \end{equation} 
  then there exist $c_7-c_9$ such that
  \begin{equation}\label{densityversion}
   c_7 t^{-d/2} \left( e^{ \frac{-c_8 |x|^2}{t}  } 
                + e^{-c_9 |x|^\beta} \right)  \leq p_t(x) \leq c_1 t^{-d/2} \left( e^{ \frac{-c_2|x|^2}{t}  } 
                + e^{\frac{-m|x|^\beta}{2\cdot 4^\beta}} \right), \quad x\in\Rd,\, t> t_0.
  \end{equation}
\end{twierdzenie}

The last application is given in Section \ref{section_high}. We do not
assume here anything (except of \eqref{Levy1}) on the behavior of $\nu$ outside of the ball $B(0,1)$ and
we consider the processes with high intensity of small jumps, i.e., such that $ \nu(dx) \asymp |x|^{-d-2} \left[\log\left(\frac{2}{|x|}\right)\right]^{-\beta}\, dx,$ for $|x|<1$,
where $\beta > 1 $. We extend the results obtained previously in \cite{Mimica1} 
and \cite{KSz2}. In this case sharp estimates for large times were already known. We investigate here the difficult case
of small $t$ and using Lemma \ref{below_exp2b} we get a new lower bound. This estimate seems to be optimal
in view of new results obtained in \cite{Mimica_new} for a particular case of subordinated Brownian motion.

Let us also mention other related results. Estimates of transition densities for stable L\'evy processes has been studied, e.g., in \cite{BG60, PT69, H94, H03, G93, GH93, D91, W07, BS2007}. Recent papers \cite{S10, S11, KnopKul, KSch2012, KSz1, Knop13} contain the estimates for more general classes of L\'evy processes, including tempered processes with intensities of jumps lighter than polynomial. The paper \cite{BGR13} deals with estimates of densities for isotropic unimodal L\'evy processes, while the papers \cite{Mimica1, KSz1} discusses the processes with higher intensity of small jumps, remarkably different than stable one. In \cite{ChKimKum1, ChKum08, KSz1} the authors investigate the case of more general, non-necessarily space homogeneous, symmetric jump Markov processes with jump intensities dominated by those of isotropic stable processes. Estimates of kernels for processes which are solutions of SDE driven by L\'evy processes were obtained in \cite{Picard97}. For estimates of derivatives of L\'evy densities we refer the reader to \cite{S10a, BJ07, SchSW12, KSz1, KR13, Knop13}. In \cite{JKLSch2012} an interesting geometric interpretation of the transition densities for symmetric L\'evy processes was given.

\section{Preliminaries}

For a set $A\subset\Rd$ we denote $\delta(A)=\dist(0,A)=\inf\{|y|:\:y\in A\}$ and
$\diam(A)=\sup\{|y-x|:\:x,y\in A\}$. By ${\mathcal{B}}(\Rd )$ we denote Borel sets in $\Rd$.

General estimates of the densities at the origin were obtained in \cite{KSz2}. It follows from Lemma 6 and 7 in \cite{KSz2} that if
\eqref{Assum} holds then there exist constants $c_1=c_1(d),c_2=c_2(d,M_0),\theta=\theta(d,M_0)$ such that
\begin{equation}\label{eq:p_est_below1}
  c_1 \left(h(t)\right)^{-d} \leq p_t(x) \leq c_2 \left(h(t)\right)^{-d}\quad  \mbox{for}\quad |x|< \theta h(t),\, t\in (0,t_p).
\end{equation}

Lower estimates of densities by the L\'evy measure were also obtained in \cite{KSz2}. 
We include here a modified version of Theorem 2 of \cite{KSz2}. The proof differs only in a few details
and we give it in the Appendix.

\begin{prop}\label{th:p_est_below}
If \eqref{Assum} holds
then for every $\eta>0$ there exist constants $L_1=L_1(d,\eta,M_0)$, $L_2=L_2(d,\eta,M_0)<\eta$ 
such that
\begin{equation}\label{eq:p_est_below0}
  p_t(x) \geq L_1 t\left(h(t)\right)^{-d} \nu (B(x,L_2 h(t)))\quad  
  \mbox{for}\quad  |x|\geq \eta h(t),\, t\in (0,t_p).
\end{equation}
\end{prop}

The following proposition was proved in \cite{KSz2}, Theorem 1.
\begin{prop}\label{prop1} Assume that
$\nu$ is a symmetric L\'evy measure such that
\begin{equation}\label{eq:nu_estim}
  \nu(A) \leq M_3 f(\delta(A))[\diam(A)]^{\gamma},\quad A\in{\mathcal{B}}(\Rd ),
\end{equation}
where $\gamma\in[0,d]$, and $f:\:[0,\infty)\to [0,\infty]$ is nonincreasing function satisfying
\begin{equation}\label{eq:tech_assumpt}
  \int_{|y|>r} f\left(s\vee |y|-\frac{|y|}{2} \right) \,\nu(dy) 
  \leq 
  M_4 f(s) \Psi\left(\frac{1}{r} \right),\quad s>0,r>0,
\end{equation}
for some constants $M_3,M_4>0$.
If \eqref{Assum} holds then 
there exist constants
$c_1=c_1(d,M_0,M_3,M_4),c_2=c_2(d,M_0),c_3=c_3(d,M_0)$ such that
\begin{eqnarray*}
  p_t(x) 
  & \leq & c_1 \left(h(t)\right)^{-d} \min\left\{ 1, t\left[h(t)\right]^{\gamma}
           f\left(|x|/4\right)
          + \, e^{-c_2 \frac{|x|}{h(t)}\log\left(1+\frac{c_3|x|}{h(t)}\right)}
            \right\},\\
   &     & x\in\Rd,\, t\in (0,t_p).
\end{eqnarray*}
\end{prop}

\section{Lower estimates}\label{section_lower}

The following lemma contains a lower estimate of densities $p_t$ in terms of
a function $F$ which is a lower bound for $\Psi(s)/s$. 

\begin{lemat}\label{below_exp2a}
Assume that \eqref{Assum} holds and there exists a strictly increasing continuous function $F: [0,\infty)\to [0,\infty)$ such that $F(0)=0$, 
$\lim_{s\to\infty} F(s)=\infty$ 
and 
\begin{equation}\label{condition1}
  F(s)\leq \frac{\Psi(s)}{s},\quad \mbox{for} \quad s\in [0,\infty).
\end{equation}
Then there are constants $c_i=c_i(d,M_0)$, $i=1,2$, such that
for every $\eta\in(0,\theta)$, where $\theta$ is the constant from \eqref{eq:p_est_below1}, we have
 
\begin{displaymath}
  p_t(x) \geq c_1 h(t)^{-d} e^{-c_2 |x| F^{-1}(2|x|/(\eta t))/\eta},  
\end{displaymath}
for $t\in \left(0, t_p \right)$ and $x\in\Rd$.
\end{lemat}

We may also weaken the assumptions obtaining estimates on smaller domain. We give only the proof of Lemma 
\ref{below_exp2b} since the proof of Lemma \ref{below_exp2a} differs only in few details 
(one can just put $s_0=\infty$ here).

\begin{lemat}\label{below_exp2b}
Assume that \eqref{Assum} holds and there exists a constant $s_0\in (0,\infty)$ and a strictly increasing continuous function $F: [0,s_0] \to [0,\infty)$ such that $F(0)=0$, 
and 
\begin{equation}\label{condition}
  F(s)\leq \frac{\Psi(s)}{s},\quad \mbox{for} \quad s\leq s_0.
\end{equation}
Then there are constants $c_i=c_i(d,M_0)$, $i=1,2$, such that
for every $\eta\in(0,\theta)$, where $\theta$ is the constant from \eqref{eq:p_est_below1}, we have
 
\begin{equation}\label{eq:p_est_below2}
  p_t(x) \geq c_1 h(t)^{-d} e^{-c_2 |x| F^{-1}(2|x|/(\eta t))/\eta},  
\end{equation}
for $t\in \left( \frac{1}{s_0F(s_0/2)}, t_p \right)$ and $|x| < \frac{\eta t F\left(s_0/2\right)}{2}$.
\end{lemat}
\begin{proof}
  Let $t\in \left( \frac{1}{s_0F(s_0/2)}, t_p \right)$, $\eta\in (0,\theta)$ and $x\in\Rd$ be such that $0<|x|<\eta t F(s_0/2)/2$. We first assume that
  $\frac{4|x|}{\eta}F^{-1}\left(\frac{2|x|}{\eta t}\right)\geq 1$ and let $n\in\N_0$ be such that
  $$
     2^n \leq \frac{4|x|}{\eta}F^{-1}\left(\frac{2|x|}{\eta t}\right) < 2^{n+1}.
  $$
  We have $F^{-1}(\frac{2|x|}{\eta t}) < \frac{\eta 2^{n+1} }{4|x|}$, and by  
  \eqref{condition} we obtain
  \begin{displaymath}
	  \frac{2|x|}{\eta t} < F\left(\frac{\eta 2^{n+1} }{4|x|}\right) 
	  = F\left(\frac{\eta 2^n }{2|x|}\right)\leq \Psi\left( \frac{\eta 2^n}{2|x|} \right)\frac{2|x|}{\eta 2^n},
  \end{displaymath}
  since $\frac{\eta 2^n}{2|x|} \leq 2 F^{-1}\left(\frac{2|x|}{\eta t}\right)<s_0$, hence $\frac{2^n}{t} < \Psi(\frac{\eta 2^n}{2|x|})$ and $\Psi^{-1}(\frac{2^n}{t})\leq \frac{\eta 2^n}{2|x|}$ which gives
  $$
    \frac{|x|}{2^n} \leq \frac{1}{2}\eta h(t/2^n) \leq \frac{1}{2} \theta h(t/2^n).
  $$
  
   Let $k=2^n$.
   It follows from (\ref{eq:p_est_below1}) that
  \begin{equation}\label{pom1}
    p_{t/k}(y) \geq c_1 h\left(t/k\right)^{-d},\, \text{ for }\,  |y|<\theta h\left(t/k\right).
  \end{equation}
  Having the above preparation we can use now the standard method which was used, e.g., in
  the proof of Theorem 3.6 in \cite{ChKimKum2}.
  Let $0=x_0,x_1,...,x_{k-1},x_{k}=x$ be such that $x_i=(i/k)x$. We have 
  $|x_{i+1}-x_i|=\frac{|x|}{k}\leq \frac{1}{2}\eta h(t/k) \leq \frac{1}{2} \theta h(t/k)$.
  Let $B_i=B(x_i,\frac{\theta}{4}h(t/k))$.
  Using the semigroup property of $p_t$ and \eqref{pom1} we get
  \begin{eqnarray*}
    p_t(x) 
    &  =   & \int ... \int p_{t/k}(y_1) p_{t/k}(y_2-y_1)...p_{t/k}(x-y_{k-1})\, dy_1dy_2...dy_{k-1} \\
    & \geq & \int_{B_1} ... \int_{B_{k-1}} p_{t/k}(y_1) p_{t/k}(y_2-y_1)...p_{t/k}(x-y_{k-1})\, dy_1dy_2...dy_{k-1} \\
    & \geq & \left(c_1 h\left(t/k\right)^{-d}\right)^{k} \left(\omega_d \left(\frac{\theta}{4}h(t/k)\right)^d\right)^{k-1} \\
    &   =  & h(t/k)^{-d} \left(c_1\omega_d \left(\frac{\theta}{4}\right)^d\right)^k \left(\omega_d \left(\frac{\theta}{4}\right)^d\right)^{-1} \\
    &  =   & c_2 h(t/k)^{-d} e^{-c_3 k} \geq c_2 h(t)^{-d} e^{-4c_3 |x| F^{-1}(2|x|/(\eta t))/\eta},
  \end{eqnarray*}
  where $c_1=c_1(d)$ and $c_2=c_2(d,M_0),c_3=c_3(d,M_0)$.
  Let now $\frac{4|x|}{\eta}F^{-1}\left(\frac{2|x|}{\eta t}\right) < 1$. The function 
  $g(s) = \frac{4s}{\eta}F^{-1}\left(\frac{2s}{\eta t}\right)$ is strictly
  increasing, continuous and $\sup_{s\in [0,\eta t F(s_0)/2]}g(s)= g(\eta t F(s_0)/2) = 2ts_0F(s_0)> 1$ (since $t>\frac{1}{s_0F(s_0/2)}$) 
  so there exists $s_1 > |x|$ such that $\frac{4s_1}{\eta}F^{-1}\left(\frac{2s_1}{\eta t}\right)=1$. Furthermore, 
  using the fact that $t>\frac{1}{s_0F(s_0/2)}$
  we obtain
  $$
    \frac{4s_1}{\eta}F^{-1}\left(\frac{2s_1}{\eta t}\right)
    =   1
    =   \frac{2}{s_0}F^{-1}\left(\frac{s_0 F(s_0/2)}{s_0 }\right)
    \geq \frac{2}{s_0}F^{-1}\left(\frac{1}{s_0 t}\right)
    =    \frac{4\frac{\eta}{2s_0}}{\eta}F^{-1}\left(\frac{2\frac{\eta}{2s_0}}{\eta t}\right)      
  $$ which yields $\frac{\eta}{2s_1} <s_0$.  We have then
  $\frac{2s_1}{\eta t}=F\left(\frac{\eta}{4s_1}\right)<F\left(\frac{\eta}{2s_1}\right)\leq \Psi\left(\frac{\eta}{2s_1}\right)\frac{2s_1}{\eta}$, hence
  $s_1\leq \frac{1}{2}\eta h(t)$ and \eqref{eq:p_est_below2} in this case follows directly from (\ref{eq:p_est_below1}), since $|x| < s_1$.
\end{proof}

\begin{proof}[Proof of Theorem \ref{main1}]
It follows from \eqref{eq:PsiH} that for every $s_0>0$ we have
\begin{displaymath}
	\Psi(s) \geq s^2 \int_{|y|<1/s} |y|^2 \,\nu(dy) \geq c_1 s^2,
\end{displaymath}
for $s<s_0$, where $c_1=\int_{|y|<1/s_0} |y|^2\, \nu(dy) >0 $ since $\nu(\Rd)=\infty$. Therefore
we have $\Psi(s)\geq s F(s)$ for $s<s_0$ and $F(s) = c_1 s$ and Lemma \ref{below_exp2b} yields
\begin{displaymath}
  p_t(x)	\geq c_2 h(t)^{-d} e^{-c_3|x|^2/ t},
\end{displaymath}
for $t\in (\frac{2}{c_1 s_0^2},\infty)$ and $|x|<c_4 s_0 t$.
\end{proof}

In \cite[Lemma 2]{S11} we obtained an upper estimate of densities for infinitely divisible distributions having L\'evy measures with bounded support. Here we prove the opposite bound which holds also for more general class of processes.

\begin{lemat}\label{pbelow_expxlog}
  Assume that \eqref{Assum} holds and there exist $r_0>0$ and $\kappa_0>0$ such that
\begin{equation}\label{eq:>xlog.a1}
  \inf_{x\in B(0,r_0)} \frac{d\nu_c}{dm}(x) = \kappa_0 > 0,
\end{equation}
where $\frac{d\nu_c}{dm}$ denotes the Radon-Nikodym derivative of the absolutely continuous part $\nu_c$ of $\nu$
with respect to the Lebesgue measure.
Then for every $\eta >0$ there exists $c_1=c_1(d),c_2=c_2(d,\eta,M_0)$ such that
\begin{displaymath}
  p_t(x) \geq c_1 r_0^{-d} \exp\left\{-\frac{2|x|}{r_0}\log\left(\frac{c_2 |x|}{r_0^{d+1}\kappa_0 t}\right)\right\},
\end{displaymath}
for $t\in(0,t_p)$, and $|x|\geq\max\left\{ r_0, \Psi\left(\frac{6 \eta}{r_0}\right)\frac{3r_0}{4}\, t \right\}$.
\end{lemat}
\begin{proof}
  Let 
  $$
    n=\left\lfloor \frac{4|x|}{3r_0} \right\rfloor + 1.
  $$
  We note that $ 4|x|/(3r_0) < n \leq 7|x|/(3r_0) $, since $|x|/r_0 \geq 1$.
  Furthermore, $\frac{n}{t} \geq \Psi\left(\frac{6\eta}{r_0}\right)\frac{3n r_0}{4|x|}> \Psi\left(\frac{6\eta}{r_0}\right)$. It follows that
  $ \Psi^{-1}(n/t)\geq \frac{6\eta}{r_0}$ and $\frac{r_0}{6}\geq \eta h(\frac{t}{n})$.
  From \eqref{eq:>xlog.a1} and Proposition \ref{th:p_est_below} for every $s\in (0,t_p)$ such that $r_0>\eta h(s)$ we get
  $$
    p_s(y) \geq c_1 s \kappa_0,\quad \eta h(s) \leq |y| < r_0,
  $$
  for some constant $c_1=c_1(d,\eta,M_0)$.
  Let $0=x_0,x_1,...,x_{n-1},x_{n}=x$ be such that $|x_{i+1}-x_i|=\frac{|x|}{n}$ (taking $x_i=(i/n)x$)
  and let $B_i=B(x_i,\frac{r_0}{8})$. We note that $ \frac{3}{7}r_0 \leq \frac{|x|}{n} < \frac{3}{4}r_0$ and for $y_i\in B_i$, $y_{i+1}\in B_{i+1}$ we have $  \eta h(\frac{t}{n}) \leq \frac{r_0}{6} < |y_{i+1}-y_i| < r_0 $.  We obtain
  \begin{eqnarray*}
    p_t(x) 
    &  =   & \int ... \int p_{t/n}(y_1) p_{t/n}(y_2-y_1)...p_{t/n}(x-y_{n-1})\, dy_1dy_2...dy_{n-1} \\
    & \geq & \int_{B_1} ... \int_{B_{n-1}} p_{t/n}(y_1) p_{t/n}(y_2-y_1)...p_{t/n}(x-y_{n-1})\, dy_1dy_2...dy_{n-1} \\
    & \geq & \left(c_1 \frac{t}{n}\kappa_0 \right)^{n} \left(\omega_d \left(\frac{r_0}{8}\right)^d\right)^{n-1} 
              =  \left(c_1 \frac{t}{n}\kappa_0 \omega_d \left(\frac{r_0}{8}\right)^d \right)^{n} 
                 \left(\omega_d \left(\frac{r_0}{8}\right)^d\right)^{-1}\\
    &   =  &  c_2 r_0^{-d} \left(c_3 \frac{t}{n}\kappa_0 r_0^d \right)^{n} = 
              c_2 r_0^{-d} \exp\left\{ -n \log\left( \frac{n}{c_3 t \kappa_0 r_0^d}\right)\right\}\\
    & \geq   & c_2 r_0^{-d} \exp\left\{ -\frac{2|x|}{r_0} \log\left( \frac{2|x|}{c_3 t \kappa_0 r_0^{d+1}}\right)\right\},
  \end{eqnarray*}
  with $c_2=c_2(d), c_3=c_3(d,\eta,M_0)$, and in the last line we use that $n\leq 2|x|/r_0$,
  since $\lfloor \frac{4}{3} u \rfloor+1 \leq \frac{4}{3} u + \frac{2}{3} u = 2u$ for $u\geq \frac{3}{2}$,
  and $\lfloor \frac{4}{3} u \rfloor+1 = 2 \leq 2u$ if $u\in [1,\frac{3}{2})$.
\end{proof}

\begin{proof}[Proof of Theorem \ref{main2}.] It follows directly from Lemma \ref{pbelow_expxlog}.
\end{proof}

\section{Application to L\'evy measures with bounded support}\label{app_truncated}

\subsection{General case}

Using above lemmas we obtain the following lower estimate of densities for semigroups with truncated L\'evy measures. 

\begin{twierdzenie}\label{tw:zdolu_truncated}
  Assume that there exists constant $r_0 >0$ such that $\supp \nu \subset B(0,r_0)$ and
\begin{equation}\label{eq:ikrd}
  \kappa_0 = \inf_{0<|x|<r_0} \frac{d\nu_c}{dm}(x) > 0,
\end{equation}
  If \eqref{Assum} holds with $t_p=\infty$ then there exist constants 
  $c_i=c_i(d,M_0)$, $i=1,2,3,4$, such that
  \begin{displaymath}
    p_t(x) \geq c_1 \left\{
    \begin{array}{lll}
      h(t)^{-d} & \mbox{for} & |x| \leq \eta_0 h(t),\,t>0,\\
      t (h(t))^{-d} \nu(B(x,c_2 h(t))) & \mbox{for} & \eta_0 h(t)\leq |x| \leq r_0 ,\,t\leq t_0, \\
      h(t)^{-d} \exp\left\{-\frac{ c_3 |x|^2}{ m_0 t}\right\}, & \mbox{for} & \eta_0 h(t)\leq |x|\leq C_* t,\, t\geq t_0, \\
      r_0^{-d}\exp\left\{-\frac{2 |x|}{r_0} \log\left(\frac{c_4 |x|}{r_0^{d+1}\kappa_0 t}\right)\right\}, & \mbox{for} & |x| \geq r_0\vee C_*t,\, t>0,
    \end{array}
    \right.
  \end{displaymath}
where $m_0=\int |y|^2\, \nu(dy)$, $\eta_0:=\theta \wedge \frac{L_0}{216}\wedge 1$, $t_0=\frac{4r_0^2}{\eta_0 L_0 m_0}$ and
  $C_*=\frac{\eta_0 L_0 m_0}{4r_0}$. We also have
\begin{equation}\label{eq:Lh2}
	\sqrt{L_0 m_0} \sqrt{t} \leq h(t) \leq \sqrt{2 m_0} \sqrt{t}\quad \text{ for}\quad t > \frac{r_0^2}{L_0 m_0}.
\end{equation}
\end{twierdzenie}

\begin{proof}
The first estimate follows from \eqref{eq:p_est_below1} and the second from Proposition \ref{th:p_est_below}.

Using \eqref{eq:PsiH} we obtain
\begin{equation}\label{eq:LP2}
	L_0 m_0 r^2 \leq \Psi(r) \leq 2 m_0 r^2,\quad r\leq \frac{1}{r_0},
\end{equation}
since $H(r)=r^2 \int |y|^2\,\nu(dy) = m_0 r^2,$ for $r\leq 1/r_0$.
For $t > r_0^2/(L_0 m_0) \geq  1/\Psi(1/r_0)$ we have $h(t)\geq r_0$
and \eqref{eq:Lh2} follows by taking $r=1/h(t)$ in \eqref{eq:LP2}.

Choosing $F(s)=L_0 m_0 s$ and $s_0=1/r_0$ in Lemma \ref{below_exp2b} we obtain
\begin{displaymath}
  p_t(x) \geq c_1 h(t)^{-d} \exp\left\{-\frac{2 c_2 |x|^2}{\eta_0^2 L_0 m_0 t}\right\},  
\end{displaymath}
for $t>t_0\geq \frac{2r_0^2}{ L_0 m_0}$ and $|x| < \frac{\eta_0 L_0 m_0  }{4 r_0} t$.
From \eqref{eq:ikrd} and Lemma \ref{pbelow_expxlog} we get
$$
  p_t(x) \geq c_3 r_0^{-d} \exp\left\{-\frac{2|x|}{r_0}\log\left(\frac{c_4 |x|}{r_0^{d+1} \kappa_0 t}\right)\right\},
$$
for $|x|\geq\max\left\{ r_0, C_* t\right\}\geq \max\left\{ r_0, \frac{54 \eta_0^2 m_0}{r_0}\, t \right\}$. 
\end{proof}

Now we deal with upper bounds. 
In the following lemma we improve the estimates obtained previously in \cite[Lemma 2]{S11}. We use here in essential way the results of 
\cite{KSch2012}.

\begin{lemat}\label{th:p_est_above}
  Assume that the L\'evy measure $\nu$ is symmetric and $P_t$ has a transition density $p_t$ for all $t>0$.
  If, for some $r_0>0$, we have $\supp\nu\subset B(0,r_0)$ then  
  \begin{equation}\label{kwadratbelow}
    p_t(x)\leq e^{\frac{-|x|}{4r_0}\log\left(\frac{r_0|x|}{2tm_0}\right)}p_t(0),\quad |x|\geq\frac{2em_0}{r_0}\, t,
  \end{equation}
  where $m_0=\int|y|^2\,\nu(dy)$. If additionally there exist constants $M_5,M_6>0$ such that
  \begin{equation}\label{eq:1}
    \int |y|^2 e^{|\xi||y|}\, \nu(dy) \leq M_5,\quad |\xi|\leq M_6,
  \end{equation}
  then
  $$
    p_t(x)\leq e^{-\frac{|x|^2}{4tM_5}}p_t(0),\quad |x|\leq 2M_5 M_6 t.
  $$  
\end{lemat}

\begin{proof}
  We use Theorem 6 of \cite{KSch2012} obtaining
  $$
    p_t(x) \leq e^{-D^2_t(x)}p_t(0),\quad x\in\Rd,t>0,
  $$
  where
  $$
    D^2_t(x) = -v_t(\xi_0,x),
  $$
  $$
    v_t(\xi,x) = -\scalp{\xi}{x} + t \int \left(\cosh(\scalp{\xi}{y})-1\right)\,\nu(dy),
  $$
  and $\xi_0=\xi_0(t,x)\in\Rd$ is such that $v_t(\xi_0,x)=\inf_{\xi\in\Rd} v_t(\xi,x)$.
  We have $\cosh(s)-1\leq s^2 e^s$ for all $s>0$, therefore
  $$
    v_t(\xi,x) \leq -\scalp{\xi}{x} + t |\xi|^2 \int |y|^2 e^{|\xi||y|}\,\nu(dy) \leq -\scalp{\xi}{x} + t |\xi|^2 e^{|\xi|r_0} m_0.
  $$
  We choose $s>0$ such that $se^{s r_0}=\frac{|x|}{2tm_0}$. If $\frac{|x|}{2tm_0}\geq \frac{e}{r_0}$ then
  $\frac{1}{2 r_0}\log(\frac{r_0|x|}{2tm_0}) \leq s\leq \frac{1}{r_0}\log(\frac{r_0|x|}{2tm_0})$, 
  since $e^u\leq ue^u \leq e^{2u}$ for $u\geq 1$. Taking $\xi_1=\frac{sx}{|x|}$ 
  we obtain
  \begin{displaymath}
    v_t(\xi_0,x)\leq v_t(\xi_1,x)\leq -\frac{1}{2}s|x| \leq \frac{-|x|}{4 r_0}\log\left(\frac{r_0|x|}{2tm_0}\right),
  \end{displaymath}
  and \eqref{kwadratbelow} follows. 
  If (\ref{eq:1}) is satisfied then
  \begin{displaymath}
    v_t(\xi,x) \leq -\scalp{\xi}{x} + t |\xi|^2 \int |y|^2 e^{|\xi||y|}\,\nu(dy) \leq -\scalp{\xi}{x} + t |\xi|^2 M_5,
  \end{displaymath}
  for $|\xi|\leq M_6$. Taking $\xi_2=\frac{1}{2tM_5}x$ we obtain 
  \begin{displaymath}
    v_t(\xi_0,x)\leq v_t(\xi_2,x)\leq \frac{-|x|^2}{4tM_5},
  \end{displaymath}
  for $|x|\leq 2M_5M_6\, t$.
\end{proof}

We summarize estimates obtained in Lemma \ref{th:p_est_above}, Proposition \ref{prop1} and \eqref{eq:p_est_below1} in
the following Theorem. We recall that the first estimate holds in fact for every process satisfying \eqref{Assum}.
We note also that for $t>1$  we have $h(t)\asymp \sqrt{t}$ and the exponential term in the second inequality below dominates the forth bound so for $|x| > C^* t$ the latter estimate is more exact. Similarly the third estimate is more exact then the second for $|x|<C^* t$ and $t>1$. For small $t$
the result of analogous comparison depends on functions $h$ and $f$.
We will compare the bounds more precisely in the next section under additional assumptions on $\nu$.

\begin{twierdzenie}\label{tw:zgory_truncated}
  Assume that \eqref{Assum} holds with $t_p=\infty$, $\supp(\nu) \subset B(0,r_0)$ for some $r_0>0$ and there exist a constant $\gamma$ and a nonincreasing function $f$ such that
  \eqref{eq:nu_estim} and \eqref{eq:tech_assumpt} hold. 
  Then there exist constants $\theta=\theta(d,M_0),c_1=c_1(d,M_0,M_3,M_4),$
  $c_2=c_2(d,M_0), c_3=c_3(d,M_0)$ such that
  \begin{displaymath}
    p_t(x) \leq c_1 \left\{
    \begin{array}{lll}
      h(t)^{-d} & \mbox{for} & |x|\leq \theta h(t),\\
      t\left[h(t)\right]^{\gamma-d}
           f\left(|x|/4\right)
          + \, h(t)^{-d }e^{-c_2 \frac{|x|}{h(t)}\log\left(1+\frac{c_3|x|}{h(t)}\right)} & \mbox{for} & |x| \geq \theta h(t) \\
      h(t)^{-d} \exp\left\{-\frac{ |x|^2}{ 4em_0\, t}\right\}, & \mbox{for} & |x|\leq C^* t,\\
      h(t)^{-d} \exp\left\{-\frac{|x|}{4r_0} \log\left(\frac{r_0 |x|}{2m_0\, t}\right)\right\}, & \mbox{for} & |x| \geq C^*t,
    \end{array}
    \right.
  \end{displaymath}
  where $C^* = \frac{2em_0}{r_0},$ and $m_0=\int|y|^2\,\nu(dy)$.
\end{twierdzenie}

\begin{proof}
The first inequality follows from \eqref{eq:p_est_below1} and the second from Proposition \ref{prop1}.
It follows from Lemma \ref{th:p_est_above} and \eqref{eq:p_est_below1} that
\begin{displaymath}
  p_t(x)\leq c_1 h(t)^{-d} e^{\frac{-|x|}{4r_0}\log(\frac{r_0|x|}{2tm_0})},\quad |x|\geq\frac{2em_0}{r_0}\, t.
\end{displaymath}
Taking $M_6=\frac{1}{r_0}$ and $M_5=em_0$ in \eqref{eq:1} we get
\begin{displaymath}
  p_t(x)\leq c_1 h(t)^{-d} e^{-\frac{|x|^2}{4tem_0}},\quad |x|\leq \frac{2e m_0}{r_0 } t.	
\end{displaymath}
\end{proof}

\subsection{Absolutely continuous L\'evy measures}\label{section_abscont}

In this section we always assume that the L\'evy measure $\nu$ is absolutely continuous with respect to the Lebesgue measure
on $\Rd\setminus\{ 0 \}$ with a density $\bar\nu$. Moreover we assume that there exist constants $m_1,m_2,r_0>0$ and
a nonincreasing function $f: (0,r_0] \to (0,\infty) \ $ such that
\begin{equation}\label{eq:compar}
     m_1 f(|x|) \leq \bar\nu(x) \leq m_2 f(|x|),\quad \text{for} \quad |x| < r_0.
\end{equation}
Apart from Lemma \ref{l:doubling} we consider here $\nu$ such that $\supp(\nu)\subset B(0,r_0)$.

We will also assume that $f$ satisfies \eqref{Assum2}.
We note that for a nonincreasing functions \eqref{Assum2} yields the following both side doubling property
\begin{equation}\label{eq:doubledoubling}
	c_1 f(r) \leq f(2r) \leq c_2 f(r),\quad 2r \leq r_0,
\end{equation}
for some constants $c_1,c_2 > 0$. If $f$ is nonincreasing and 
\eqref{eq:doubledoubling} holds with $2^{-d-2} < c_1 \leq c_2 < 2^{-d}$ then
\eqref{Assum2} is satisfied.

The condition \eqref{Assum2} holds for many typical functions $f$ such as $f(s) = s^{-d-\alpha}$, $\alpha\in (0,2)$,
or $f(s)=s^{-d-\alpha} (\log (1+\frac{1}{s}))^{-\beta}$, $\beta\in\R$. However we note that a density $\bar\nu$ of a L\'evy measure not always satisfies the doubling property at the origin. 
For example we can observe it
for 
\begin{displaymath}
  \bar\nu(x) = \frac{2^{(2+d)k^2}}{k^2+1},\quad \mbox{ for } \quad 2^{-(k+1)^2} <  |x|\leq 2^{-k^2},
\end{displaymath}
since
\begin{displaymath}
	\frac{f(2\cdot 2^{-k^2})}{f(2^{-k^2})} = \frac{k^2+1}{(k-1)^2+1}2^{(2+d)(-2k+1)} \to 0,\quad 
	\mbox{ as } \quad k\to\infty. 
\end{displaymath}

The constants appearing in this section can
all depend on $m_1,m_2,\beta_1,\beta_2,M_1,M_2,r_0,f$ and $\nu$ and we will not mention it explicitly below.
We will use the notation $f\asymp g$ to indicate that there exist constant $c_1,c_2$ such that
$c_1 g \leq f\leq c_2 g$. 
In the following lemmas we obtain some 
interesting properties of semigroups satisfying above conditions.  

\begin{lemat}\label{l:doubling} Assume that the L\'evy measure $\nu$ 
satisfies \eqref{eq:compar} and the function $f$ satisfies \eqref{Assum2}. If $\kappa=\inf_{|x|\leq r_0} f(x)>0$ 
  then
  \begin{equation}\label{eq:nu=Psi/xd}
    \bar\nu(x) \asymp \frac{\Psi(1/|x|)}{|x|^d}, \quad |x|<r_0.
  \end{equation}
\end{lemat}

\begin{proof} 
From \eqref{eq:PsiH} we get
\begin{displaymath}
  \Psi(1/r) \asymp r^{-2} \int_{|y|<r} |y|^2 \bar\nu(y)\, dy + \int_{|y|\geq r} \bar\nu(y)\, dy.
\end{displaymath}
We observe that
\begin{displaymath}
	\int_{|y|<r} |y|^2 \bar\nu(y)\, dy  \geq m_1 f(r) \int_{|y|<r} |y|^2  \, dy 
	= \frac{m_1 d\omega_d}{d+2} r^{d+2} f(r), \quad r<r_0,
\end{displaymath}
where $\omega_d$ denotes the volume of the unit ball in $\Rd$, and the lower estimate in \eqref{eq:nu=Psi/xd} follows. 

Now we prove the upper bound. For $r\leq r_0$, using \eqref{eq:compar} and \eqref{Assum2}, we obtain
\begin{displaymath}
  \int_{|y|<r} |y|^2 \bar\nu(y)\, dy 
  \leq \int_{|y|<r} |y|^2 m_2 M_2 \left(\frac{r}{|y|}\right)^{\beta_2} f(r)\, dy 
   =    \frac{d\omega_d m_2 M_2}{2+d-\beta_2}   r^{d+2} f(r).
\end{displaymath}

Furthermore
\begin{eqnarray*}
  \int_{r \leq |y| \leq r_0} \bar\nu(y)\, dy
  & \leq & m_2 d\omega_d \int_r^{r_0} s^{d-1} f(s)\, ds  \leq \frac{m_2 d\omega_d r^{\beta_1} f(r)}{M_1} \,\int_r^{r_0} s^{d-1-\beta_1} \, ds\\
  & \leq & c_1 r^d f(r),
\end{eqnarray*}
and from \eqref{Assum2} we obtain $r^df(r)\geq M_1 r_0^{\beta_1} r^{d-\beta_1} f(r_0) \geq M_1 r_0^d \kappa$, for $r<r_0$,
and this yields
\begin{displaymath}
  \int_{|y|\geq r} \bar\nu(y)\, dy \leq c_1 r^d f(r) + \int_{r_0 < |y| } \bar\nu(y)\, dy \leq 
  \left( c_1 + \frac{\int_{r_0 < |y| } \bar\nu(y)\, dy}{M_1r_0^d\kappa} \right) \,r^d f(r),
\end{displaymath}
and the lemma follows.
\end{proof}

\begin{lemat}\label{l:Assum2Assum} Assume that $\nu$ satysfies \eqref{eq:compar} and \eqref{Assum2} 
and $\supp \nu \subset B(0,r_0)$.
  If $\kappa=\inf_{|x|\leq r_0} f (x)>0$ then \eqref{Assum} holds with $t_p=\infty$.
\end{lemat}
\begin{proof} 
First we will prove that $\Psi(|\xi|)\asymp \Phi(\xi)$. The inequality $\Phi(\xi)\leq \Psi(|\xi|)$ follows directly from the definition of $\Psi$ so we have to prove only the opposite estimate.
Using Lemma \ref{l:doubling} for $|\xi|>\frac{1}{r_0}$ we get
  \begin{eqnarray*}
  \Phi(\xi) 
  &   =  & \int \left(1-\cos(\scalp{\xi}{y}) \right) \bar\nu(y)\, dy \\
  & \geq & m_1 \int \left(1-\cos(\scalp{\xi}{y}) \right) f(|y|)\, dy \\
  & \geq &	c_1 \int_{|y|< \frac{1}{|\xi|}} \left(1-\cos(\scalp{\xi}{y}) \right) \frac{\Psi\left(\frac{1}{|y|}\right)}{|y|^d} \, dy \\
  & \geq &	c_2 \Psi(|\xi|) \int_{|y|< \frac{1}{|\xi|}} \left(\scalp{\xi}{y} \right)^2 \frac{1}{|y|^d} \, dy \\
  &  =   & c_3 \Psi(|\xi|),
\end{eqnarray*}
since $\int_{|y|< \frac{1}{|\xi|}} \left(\scalp{\xi}{y} \right)^2 \frac{1}{|y|^d} \, dy=
|\xi|^2 \int_{|y|< \frac{1}{|\xi|}} \left(\scalp{\frac{\xi}{|\xi|}}{y} \right)^2 \frac{1}{|y|^d} \, dy
=|\xi|^2 \int_{|y|< \frac{1}{|\xi|}} \frac{y_1^2}{|y|^d} \, dy = const.$, where we use
the rotational invariance of the Lebesgue measure. 
This and Lemma \ref{l:doubling} yield
\begin{equation}\label{eq:Psi_f}
	\Phi(\xi) \asymp \Psi(|\xi|) \asymp f(1/|\xi|)|\xi|^{-d},\quad |\xi| >\frac{1}{r_0}.
\end{equation}
For $|\xi|\leq \frac{1}{r_0}$ we have
\begin{displaymath}
	\Phi(\xi) = \int \left(1-\cos(\scalp{\xi}{y}) \right) \bar\nu(y)\, dy \geq c_4 \int_{|y|<r_0} |\scalp{\xi}{y}|^2 f(|y|)\, dy \geq c_5 |\xi|^2.
\end{displaymath}
Further, by \eqref{eq:PsiH} we have
\begin{displaymath}
	\Psi(r) \asymp r^2 \int_{|y|<r_0} |y|^2 \, \nu(dy),\quad r\leq \frac{1}{r_0},
\end{displaymath}
and so
\begin{equation}\label{eq:Psi_2}
	\Phi(\xi) \asymp \Psi(|\xi|) \asymp |\xi|^2,\quad |\xi| \leq \frac{1}{r_0}.
\end{equation}
It follows from \eqref{eq:Psi_f},\eqref{eq:Psi_2} and \eqref{Assum2} that there exist $L>1$ and $c_*>1$ such that
\begin{displaymath}
	\Psi(Lr) \geq c_* \Psi(r),\quad r>0.
\end{displaymath}
This yields
\begin{displaymath}
	\Psi(L^n/h(t)) \geq c_*^n \Psi(1/h(t)) = \frac{c_*^n}{t}, \quad t>0,n\in\N,
\end{displaymath}
and we obtain
\begin{eqnarray*}
  \int_{\Rd} e^{-t\Re\left(\Phi(\xi)\right)}|\xi|\, d\xi
  & \leq & \int e^{-c_6 t \Psi(|\xi|)} |\xi|\, d\xi \\
  &   =  & \int_{|\xi|\leq 1/h(t)} e^{-c_6 t \Psi(|\xi|)} |\xi|\, d\xi + \int_{|\xi| > 1/h(t)} e^{-c_6 t \Psi(|\xi|)} |\xi|\, d\xi \\
  & \leq & c_7 h(t)^{-d-1} + \sum_{n=0}^\infty \int_{\frac{L^n}{h(t)}<|\xi|\leq \frac{L^{n+1}}{h(t)}} e^{-c_6 t \Psi(|\xi|)} |\xi|\, d\xi \\
  & \leq & c_7 h(t)^{-d-1} + c_8 \sum_{n=0}^\infty e^{-c_6 t \Psi(L^n/h(t))} \left(\frac{L^{n+1}}{h(t)}\right)^{d+1} \\
  & \leq & c_7 h(t)^{-d-1} + c_9 h(t)^{-d-1} \sum_{n=0}^\infty e^{-c_6 c_*^n } L^{(n+1)(d+1)} \\
  &  =   & c_{10} h(t)^{-d-1},
\end{eqnarray*}
and the lemma follows.
\end{proof}

Using the above properties we can improve now the estimates obtained previously in Theorem \ref{tw:zgory_truncated} in the general case.

\begin{twierdzenie}\label{tw:abstrunc}
  Assume that $\supp(\nu)\subset B(0,r_0)$, 
  $\nu$ is symmetric and absolutely continuous with respect to the Lebesgue measure on $\Rd\setminus\{0\}$ with
  a density $\bar\nu$ and there exists a nonincreasing function $f:[0,r_0]\to [0,\infty]$ 
  such that 
  \begin{displaymath}
	  m_1 f(|x|) \leq \bar\nu(x) \leq m_2 f(|x|),\quad 0<|x|<r_0.
  \end{displaymath}
  where $f$ satisfies \eqref{Assum2},
  and $\kappa = \inf_{s\in (0,r_0]} f(s) > 0$.
  Then there exist $c_1,c_2,c_3,c_4$, $\eta_1$ and $C^*$ such that
  \begin{equation}\label{eq:teza<}
    p_t(x) \leq c_1 \left\{
    \begin{array}{lll}
      h(t)^{-d} & \mbox{for} & |x| \leq \eta_1 h(t),\,t>0,\\
      t f(|x|) & \mbox{for} & \eta_1 h(t)\leq |x| \leq r_0 ,\,t\leq t_1, \\
      h(t)^{-d} \exp\left\{-\frac{ c_2 |x|^2}{ t}\right\}, & \mbox{for} & \eta_1 h(t)\leq |x|\leq C^* t,\, t\geq t_1, \\
      \exp\left\{-c_3 |x| \log\left(\frac{c_4 |x|}{t}\right)\right\}, & \mbox{for} & |x| \geq r_0\vee C^*t,\, t>0,
    \end{array}
    \right.
  \end{equation}
  where $t_1=r_0/C^*$.
\end{twierdzenie}

\begin{proof}
  First we prove the second inequality. Let $f_*(s)=f(s)$ for $s<r_0$ and $f_*(s)= \kappa$ for $s\geq r_0$. We have 
  $\nu(A)\leq c_1 f_*(\delta(A))) \left(\diam (A)\right)^d$ for every Borel set $A$. 
  It follows from \eqref{Assum2} that 
  $f_*\left(s\vee |y|-\frac{|y|}{2} \right) \leq f_*\left(\frac{s}{2}\right) \leq  c_2 f_*(s)$, 
  for all $y\in\Rd$, $s>0$, and \eqref{eq:tech_assumpt} holds for $f_*$ since $\nu(B(0,r)^c)\leq c_3 \Psi(1/r)$ by \eqref{eq:PsiH}. 
  Lemma \ref{l:Assum2Assum} yields that \eqref{Assum} holds and from Proposition \ref{prop1} we obtain
  \begin{equation}\label{eq:basicest}
    p_t(x) 
   \leq  c_4 \left(h(t)\right)^{-d} \min\left\{ 1, t\left[h(t)\right]^{d}
           f_*\left(|x|/4\right)
          + \, e^{-c_5 \frac{|x|}{h(t)}\log\left(1+\frac{c_6|x|}{h(t)}\right)}
            \right\},
  \end{equation}
  for all $x\in\Rd\setminus \{ 0\}$ and $t>0$. Now we will show that
  \begin{displaymath}
	   t\left[h(t)\right]^{d} f_*\left(|x|/4\right) \geq c e^{-c_5 \frac{|x|}{h(t)}\log\left(1+\frac{c_6|x|}{h(t)}\right)},
  \end{displaymath}
  for $|x|>\eta_1 h(t)$, and hence the exponential term in \eqref{eq:basicest} can be omitted.
  We observe that \eqref{eq:Psi_f} and \eqref{eq:Psi_2} yield
  \begin{displaymath}
    t = \frac{1}{\Psi(1/h(t))} \asymp \left\{
    \begin{array}{lll}
       h(t)^2 & \mbox{for} & t \geq \frac{1}{\Psi(1/r_0)}, \\
       \frac{1}{f(h(t))h(t)^{d}} & \mbox{for} & t < \frac{1}{\Psi(1/r_0)}.
    \end{array}
    \right.
  \end{displaymath}
  Therefore, for $|x|>4r_0$ and $t\geq \frac{1}{\Psi(1/r_0)} $ we have 
  $t\left[h(t)\right]^{d} f_*\left(|x|/4\right) = t\left[h(t)\right]^{d} \kappa 
  \geq (|x|/h(t))^{-d-2}|x|^{d+2}
  \geq c_7 e^{-c_5 \frac{|x|}{h(t)}\log\left(1+\frac{c_6|x|}{h(t)}\right)}$,
  and for $|x|>4r_0$ and $t\leq \frac{1}{\Psi(1/r_0)} $ by \eqref{Assum2} we have 
  $t\left[h(t)\right]^{d} f_*\left(|x|/4\right) = t\left[h(t)\right]^{d} \kappa 
  \geq c_8/f(h(t)) \geq c_9 (h(t))^{\beta_2} \geq 
    c_{10} e^{-c_5 \frac{4r_0}{h(t)}\log\left(1+\frac{c_6 4r_0}{h(t)}\right)}
  \geq c_{10} e^{-c_5 \frac{|x|}{h(t)}\log\left(1+\frac{c_6|x|}{h(t)}\right)}$.
  For $|x|\leq 4r_0$ and a constant $\eta_1<4$ from Lemma \ref{l:doubling} we get
  \begin{equation}\label{eq:fPsi2}
    f(|x|/4) \asymp \frac{\Psi(4/|x|)}{|x|^d} \geq \frac{\Psi(\eta_1/|x|)}{|x|^d}.
  \end{equation}
  We note that \eqref{eq:PsiH} yields 
\begin{equation}\label{eq:dbl_Psi}
  \Psi(2r)\leq 2 H(2r) \leq 8 H(r) \leq  \frac{8}{L_0} \Psi(r),\quad r>0. 
\end{equation}
  For $\eta_1 h(t) < |x| \leq 4r_0$ from \eqref{eq:fPsi2} and \eqref{eq:dbl_Psi} we obtain
  \begin{eqnarray*}
    th(t)^{d}f(|x|/4) 
    &   =  & \frac{h(t)^{d}f(|x|/4)}{\Psi(1/h(t))} 
          \geq c_{11} \frac{h(t)^d \Psi(\eta_1/|x|)}{\Psi(1/h(t)) |x|^d} \\
    & \geq & \frac{c_{11} L_0}{8\eta_1^d} \left(\frac{|x|}{\eta_1 h(t)}\right)^{-\log_2 \frac{8}{L_0}-d} \\
    & \geq & c_{12} e^{-c_5 \frac{|x|}{h(t)}\log\left(1+\frac{c_6|x|}{h(t)}\right)},
  \end{eqnarray*}
  hence \eqref{eq:basicest} and \eqref{Assum2} yield
  \begin{equation}\label{eq:p<L}
	   p_t(x) \leq  c_{13} t f_*\left(|x|/4\right) \leq c_{14} t f_*(|x|), \quad |x|>\eta_1 h(t),
  \end{equation}
  for every $\eta_1\in (0,4)$ and $c_{14}=c_{14}(\eta_1)$, which gives in particular the second case in \eqref{eq:teza<}.

  Now we will prove the last inequality. Lemma \ref{th:p_est_above} yields
  \begin{displaymath}
    p_t(x)\leq c_{15} h(t)^{-d} e^{\frac{-|x|}{4r_0}\log(\frac{r_0|x|}{2tm_0})},\quad 
    |x|\geq\frac{2em_0}{r_0}\, t.
  \end{displaymath}  
  From Lemma \ref{l:doubling} we have $\Psi(1/r) \asymp r^d f(r)$, and since 
  $r^{\beta_1}f(r) \geq M_1 r_0^{\beta_1}f(r_0) \geq M_1 r_0^{\beta_1}\kappa$, 
  we get $\Psi(1/r) > c_{16} r^{d-\beta_1}$ for $r\leq r_0$.
  This yields $\frac{1}{r} \geq \Psi^{-1}(c_{16}r^{d-\beta_1})$ and $h(r^{\beta_1-d}/c_{16}) \geq r$, hence we obtain
  $h(t)\geq (c_{16} t )^{1/(\beta_1-d)}$ for $t<r_0^{\beta_1-d}/c_{16}$. For $|x|> \frac{2m_0}{c_{16} r_0}$ we get
  \begin{displaymath}
	  h(t) \geq \left(\frac{2tm_0}{r_0|x|}\right)^{1/(\beta_1-d)},
  \end{displaymath}
  and this yields
  \begin{eqnarray*}
	  h(t)^{-d} \exp\left\{-\frac{|x|}{4r_0}\log\left(\frac{r_0|x|}{2tm_0}\right) \right\}
	  &  =   & h(t)^{-d} \left(\frac{2tm_0}{r_0|x|}\right)^{\frac{|x|}{4r_0}}   \\
	  & \leq & \left(\frac{2tm_0}{r_0|x|}\right)^{\frac{|x|}{4r_0}-\frac{d}{\beta_1-d}} \leq \left(\frac{2tm_0}{r_0|x|}\right)^{\frac{|x|}{8r_0} }\\
	  &   =  & \exp\left\{ -\frac{|x|}{8r_0}\log\left( \frac{r_0 |x|}{2tm_0} \right)\right\},
  \end{eqnarray*}
  provided $|x|\geq R_0=\max\{\frac{8r_0 d}{\beta_1-d},\frac{2r_0^{\beta_1-d} m_0}{c_{16} r_0},\frac{2m_0}{c_{16} r_0}\}$ and $t<r_0^{\beta_1-d}/c_{16}$.
  For $t\geq r_0^{\beta_1-d}/c_{16}$ we have $h(t)^{-d} \leq h(r_0^{\beta_1-d}/c_{16})^{-d} \leq r_0^{-d}$. We obtain
  \begin{equation}\label{eq:pelog}
	  p_t(x) \leq c_{17} e^{\frac{-|x|}{8r_0}\log(\frac{r_0|x|}{2tm_0})},\quad |x|\geq R_0\vee \frac{2em_0}{r_0}\, t,\, t>0.
  \end{equation}
  
  For $r_0\vee (\frac{2em_0}{r_0}\, t) \leq |x|\leq R_0$ we observe that
  \begin{displaymath}
	  c_{14} t f_*(|x|) = c_{14} \kappa t \leq c_{18} e^{\frac{-|x|}{R_0}\log(\frac{r_0|x|}{2tm_0})},
  \end{displaymath}
  where $c_{18}=\frac{c_{14}\kappa r_0}{2 m_0}R_0$, and the last inequality in \eqref{eq:teza<} with $C^*=\frac{2em_0}{r_0}$ follows from this, \eqref{eq:pelog} and \eqref{eq:p<L} since $\eta_1 h(t)\leq r_0$,
  for $t\leq \frac{r_0R_0}{2em_0}$ and $\eta_1=\theta \wedge \sqrt{\frac{er_0}{R_0}}\wedge 1$ (note that $\eta_1 h(t) \leq r_0$,
   provided $t\leq \frac{1}{\Psi(\eta_1/r_0)}$, 
  and that $\Psi(\eta_1/r_0)\leq 2 m_0\eta_1^2/r_0^2$ by \eqref{eq:PsiH}, since $\eta_1\leq 1$).
  
  Taking $M_6=\frac{1}{r_0}$ and $M_5=em_0$ in \eqref{eq:1} we get
  \begin{displaymath}
    p_t(x)\leq c_{19} h(t)^{-d} e^{-\frac{|x|^2}{4tem_0}},\quad |x|\leq \frac{2e m_0}{r_0 } t,
  \end{displaymath}
  which gives the third inequality in \eqref{eq:p<L}.
  The first inequality in \eqref{eq:p<L} follows from \eqref{eq:p_est_below1}, since $\eta_1 \leq \theta$.
 \end{proof} 

We can prove now Theorem \ref{tw:obustronne}.

\begin{proof}[Proof of Theorem \ref{tw:obustronne}.]
 We use here the constants $C_*,t_0,\eta_0$ from Theorem \ref{tw:zdolu_truncated}. We note that $C^* > C_*$ and $t_1 < t_0$ since $\eta_0\leq 1$ and $L_0\leq 2$. Let
 \begin{displaymath}
	 \eta_* = \eta_0 \vee \eta_1.
 \end{displaymath}
 We obtain the first estimate using \eqref{eq:p_est_below1} since $\eta_0\vee \eta_1<\theta$. 
 The inequalities in 2. and 4. follows directly from Theorem \ref{tw:zdolu_truncated} and Theorem \ref{tw:abstrunc}.
 Similarly, the third estimate for $\eta_* h(t)\leq |x| \leq C_* t$, $t\geq t_0$ is a direct consequence of these theorems. For 
 $r_0 \vee C_* t \leq |x| \leq C^* t$ and $t\geq t_1$ we have
 \begin{displaymath}
    p_t(x) \geq c_1 e^{-c_2 |x| \log(\frac{c_3|x|}{t})} \geq c_1 e^{-c_2 |x| \log(c_3 C^*)}	
    \geq c_1 h(t_1)^d h(t)^{-d} e^{-c_2 C_*^{-1} \log(c_3 C^* ) \frac{|x|^2}{t}},
 \end{displaymath}
 and for $t_1\leq t \leq t_0$ and $\eta_* h(t) \leq |x| \leq r_0$ we have
 \begin{displaymath}
	 tf(|x|) \asymp h(t)^{-d} e^{-\frac{c_4|x|^2}{t}} \asymp const. 
 \end{displaymath}
 and the inequalities in 4. follow. 
\end{proof}
 

\section{Application to tempered stable processes}\label{section_tempered}
Let
\begin{displaymath}
  \nu(A) \asymp \int_0^\infty \int_\sfera \indyk{A}(s\theta) s^{-1-\alpha} (1+s)^{\kappa} e^{-m s^\beta}\, ds\mu(d\theta),
\end{displaymath} 
where $\mu$ is bounded, symmetric and nondegenerate measure on the unit sphere $\sfera$, 
$m > 0$, $\beta\in (0,1]$, $\alpha\in (0,2)$,
$\kappa\in (-\infty,1+\alpha]$. 

We have here 
\begin{equation}\label{eq:PhiPsixi}
	\Phi(\xi) \asymp \Psi(|\xi|)\asymp |\xi|^2\wedge |\xi|^\alpha,
\end{equation}
which follows from Proposition 1 and Corollary 2 in \cite{KSz2}. We get
\begin{equation}\label{eq:hest}
	h(t)\asymp t^{1/2}\wedge t^{1/\alpha},
\end{equation}
and Lemma 5 in \cite{KSz2} yields that \eqref{Assum}
is satisfied with $t_p=\infty$. 

Such examples were discussed previously in \cite{S11},\cite{KSz2} and \cite{KSz3}.
It was proved (see also Proposition \ref{prop1} above) that if $\mu$ is a $\gamma-1$ - measure on $\sfera$, i.e., there exists a constants $c$ such that
\begin{displaymath}
	\mu\left( B(\theta,\rho) \cap \sfera \right) \leq c \rho^{\gamma -1 },\quad \theta\in\sfera,
\end{displaymath}
and that there exist $D_0\subset\sfera$ and $c>0$ such that 
\begin{displaymath}
  \mu\left( B(\theta,\rho) \cap \sfera \right) \geq c \rho^{\gamma -1 }, \quad \theta \in D_0,
\end{displaymath} 
for some $\gamma\in [1,d]$, then
\begin{eqnarray*}
  p_t(x) 
  & \leq & c_1 t^{-d/\alpha} \min\left\{ 1, t^{1+\gamma/\alpha}
           |x|^{-\gamma-\alpha}(1+|x|)^{\kappa} e^{-m|x|^\beta/4^\beta} 
          + \, e^{-c_2 t^{-1/\alpha}|x|\log\left(1+c_3t^{-1/\alpha}|x|\right)}
            \right\},\\
   & \leq & c_4 t^{-d/\alpha} \min\left\{ 1, t^{1+\gamma/\alpha}
           |x|^{-\gamma-\alpha}(1+|x|)^{\kappa} e^{-m|x|^\beta/4^\beta}
            \right\},\\
   &      & x\in\Rd,\, t\in (0,1],
\end{eqnarray*}
and
\begin{eqnarray*}
  p_t(x) 
  & \leq & c_4 t^{-d/2} \min\left\{ 1, t^{1+\gamma/2}
           |x|^{-\gamma-\alpha}(1+|x|)^{\kappa} e^{-m|x|^\beta/4^\beta} 
          + \, e^{-c_5 t^{-1/2}|x|\log\left(1+c_6 t^{-1/2}|x|\right)}
            \right\},\\
   &     & x\in\Rd,\, t\in (1,\infty).
\end{eqnarray*}
Note that we can omit the exponential term in the first estimate since 
there exists $c>0$ such that $s^{-\gamma-\alpha}(1+s)^\kappa e^{-ms^\beta/4^\beta} 
\geq c e^{-c_2s\log(1+c_3s)}$ for $s>0$ and for $t<1$ we have
$t^{1+\gamma/\alpha} |x|^{-\gamma-\alpha}(1+|x|)^{\kappa} e^{-m|x|^\beta/4^\beta}
\geq c(t^{-1/\alpha}|x|)^{-\gamma-\alpha}(1+t^{-1/\alpha}|x|)^{\kappa} 
e^{-m(t^{-1/\alpha}|x|)^\beta/4^\beta}.$ Similar procedure for large times is not possible.

More precise estimates for small $t$ were obtained in \cite{KSz3}. If $(\beta,\kappa) \in (0,1)\times (-\infty,1+\alpha]$  or $(\beta,\kappa)\in \{ 1 \}\times (-\infty,\alpha)$ then 
\begin{displaymath}
	p_t(x) \leq c_7 t^{1+\frac{\gamma-d}{\alpha}} |x|^{-\gamma-\alpha+\kappa} e^{-m|x|^\beta}, \quad t\in (0,1], |x|\geq 4,
\end{displaymath} 
and 
\begin{displaymath}
	p_t(x) \geq c_8 t^{1+\frac{\gamma-d}{\alpha}} |x|^{-\gamma-\alpha+\kappa} e^{-m|x|^\beta}, \quad t\in (0,1], x\in D,
\end{displaymath}
where $D=\{x\in\Rd:\: x=r\theta,\, r\geq 4,\, \theta\in D_0\}$.

Here we improve the estimates for large values of $t$.


\begin{proof}[Proof of Theorem \ref{t:esttemp1}.]

We will need the following preparation. As usual (see \cite{S11,KSz2,KSz3}) we divide the L\'evy measure in the two parts. For $r>0$ we denote 
$$\tilde{\nu}_r(dy)=\indyk{B(0,r)}(y)\nu(dy) \quad \text{and} \quad \nubounded{r}(dy)= \indyk{B(0,r)^c}(y)\,\nu(dy).$$
In terms of the corresponding L\'evy process, $\tilde{\nu}_r$ is related to the jumps which are close to the origin, while $\nubounded{r}$ represents the large jumps.

For the restricted L\'evy measures we consider the two semigroups of measures $\{\tilde{P}^r_t,\; t\geq 0\}$ and $\{\bar{P}^r_t,\; t\geq 0\}$ such that
$$
  \Fourier(\tilde{P}^r_t)(\xi) 
    =     \exp\left(t \int \left(e^{i\scalp{\xi}{y}}-1-i\scalp{\xi}{y}\right)
            \tilde{\nu}_r(dy)\right)\, ,  \quad \xi\in\Rd\,,
$$
and
\begin{displaymath}
  \Fourier(\bar{P}^r_t)(\xi) =
  \exp\left(t \int (e^{i\scalp{\xi}{y}}-1)\,
  \nubounded{r}(dy)\right)\, ,
  \quad \xi\in\Rd\, ,
\end{displaymath}
respectively. We have
\begin{eqnarray}\label{eq:FTildePEstimate}
  |\Fourier(\tilde{P}^r_t)(\xi)| 
  &   =  & \exp\left(-t\int_{|y|<r}
           (1-\cos(\scalp{y}{\xi}))\,
           \nu(dy)\right) \nonumber \\
  &   =  & \exp\left(-t\left(\Re(\Phi(\xi))-\int_{|y|\geq r}
           (1-\cos(\scalp{y}{\xi}))\,
           \nu(dy)\right)\right) \nonumber \\
  & \leq & \exp(-t\Re(\Phi(\xi)))\exp(2t\nu(B(0,r)^c)),
           \quad \xi\in\Rd,
\end{eqnarray}
and, therefore, by \eqref{eq:PhiPsixi}, for every $r>0$ and $t>0$ the measures $\tilde{P}^r_t$ are absolutely continuous with respect to
the Lebesgue measure with densities $\tilde{p}^r_t\in C^1_b(\Rd)$. 

We have 
\begin{align*}
  P_t=\tilde{P}^r_t \ast \bar{P}^r_t ,\quad \ \text{and} \ \quad 
  p_t= \tilde{p}^r_t * \bar{P}^r_t, \quad t>0,
\end{align*}
where 
\begin{eqnarray}\label{eq:exp}
  \bar{P}^r_t
  &  =  & \exp(t(\nubounded{r}-|\nubounded{r}|\delta_0)) =  \sum_{n=0}^\infty \frac{t^n\left(\nubounded{r}-|\nubounded{r}|\delta_0)\right)^{n*}}{n!} \\
  &  =  & e^{-t|\nubounded{r}|} \sum_{n=0}^\infty \frac{t^n\nubounded{r}^{n*}}{n!}\,,\quad t\geq 0\, . \nonumber
\end{eqnarray}
 
We will estimate first the densities $\tilde{p}^r_t\in C^1_b(\Rd)$ using Lemma \ref{th:p_est_above} and Theorem 6 
of \cite{KSch2012}.
Let
\begin{displaymath}
  \nu(A) \leq c_1 \int_0^\infty \int_\sfera \indyk{A}(s\theta) s^{-1-\alpha} (1+s)^{\kappa} e^{-m s^\beta}\, ds\mu(d\theta),
\end{displaymath} 
and $L = c_1 |\mu| \int_0^\infty s^{1-\alpha}(1+s)^{\kappa}e^{-(1-2^{-\beta/2})ms^\beta}\, ds$.
Using Lemma \ref{th:p_est_above} with 
$M_5= L$, $M_6=\frac{1}{2^{\beta/2}}mr^{\beta-1} $ we get
\begin{equation}\label{eq:etp1}
	\tilde{p}_t^r(x) \leq \tilde{p}_t^r(0) e^{-\frac{|x|^2}{4tL}},\quad |x|\leq 2^{1-\beta/2} L mr^{\beta-1}t .
\end{equation}

Recall that
Theorem 6 of \cite{KSch2012} yields
  \begin{displaymath}
    \tilde{p}_t^r(x) \leq e^{-D^2_t(x)}\tilde{p}_t^r(0),\quad x\in\Rd,t>0,
  \end{displaymath}
  where $ D^2_t(x) = -v_t(\xi_0,x) $, 
  $ v_t(\xi,x) = -\scalp{\xi}{x} + t \int \left(\cosh(\scalp{\xi}{y})-1\right)\,\tilde{\nu}_r(dy), $
  and $\xi_0=\xi_0(t,x)\in\Rd$ is such that $v_t(\xi_0,x)=\inf_{\xi\in\Rd} v_t(\xi,x)$.

If $|x|>r$ then
\begin{eqnarray*}
	\int |y|^2 e^{|\xi||y|}\, \tilde{\nu}_r(dy) 
	& \leq & \int_{|y|<|x|} |y|^2 e^{|\xi||y|}\, \nu(dy) \\
	& \leq & c_1 |\mu| \int_0^{|x|} s^{1-\alpha} (1+s)^\kappa e^{-(1-2^{-{\beta/2}})ms^\beta}\, ds \leq L, 
\end{eqnarray*}
provided $|\xi|\leq \frac{1}{2^{\beta/2}}m|x|^{\beta-1}$.

Let $R_0=\frac{2^{\beta/2}(1-2^{-{\beta/2}})}{mL}$.
Taking $\xi_1=\frac{1}{2^{\beta/2}}m|x|^{\beta-2}x$
for $|x|\geq  (t/R_0)^{\frac{1}{2-\beta}}$ 
we obtain
\begin{eqnarray*}
	v_t(\xi_1,x) \leq -\frac{1}{2^{\beta/2}}m|x|^\beta + t \frac{1}{2^{\beta}}m^2 |x|^{2(\beta-1)} L
	& \leq & \frac{1}{2^{\beta/2} } m |x|^\beta \left( -1 + \frac{1}{2^{\beta/2} } t m |x|^{\beta-2}L \right) \\
	& \leq & -\frac{1}{2^\beta}m|x|^\beta .
\end{eqnarray*}
This yields
\begin{equation}\label{eq:etp2}
	\tilde{p}^r_t(x) \leq \tilde{p}^r_t(0) e^{-\frac{1}{2^\beta}m|x|^\beta},\quad 
	|x|\geq \max\{r,\left(R_0\right)^{\frac{1}{\beta-2}} t^{\frac{1}{2-\beta}}\}.
\end{equation}

As usual, below we will use $\tilde{P}^r_t$, $\tilde{p}^r_t$ and $\bar{P}^r_t$ with $r=h(t)$ and
for simplification we will write $\tilde{P}_t=\tilde{P}^{h(t)}_t$, $\tilde{p}_t=\tilde{p}^{h(t)}_t$ and $\bar{P}_t=\bar{P}^{h(t)}_t$.

For $t$ large enough, by \eqref{eq:etp1}, \eqref{eq:etp2} and \eqref{eq:hest} we have
\begin{displaymath}
	\tilde{p}_t(x) \leq \tilde{p}_t(0) e^{-\frac{|x|^2}{4tL}},\quad |x|\leq c_2 t^{\frac{\beta+1}{2}},
\end{displaymath}
and
\begin{displaymath}
	\tilde{p}_t(x) \leq \tilde{p}_t(0) e^{-\frac{1}{2^\beta}m|x|^\beta},\quad 
	|x|\geq c_3 t^{\frac{1}{2-\beta}},
\end{displaymath}
and since $\tilde{p}_t(0) \leq c_4 h(t)^{-d}$ (see Lemma 8 in \cite{KSz2}), and for sufficiently large $t$ we have 
$c_3 t^{\frac{1}{2-\beta}} \leq c_2 t^{\frac{\beta+1}{2}}$,
we obtain
\begin{equation}\label{eq:tildepest}
	\tilde{p}_t(x) \leq \tilde{p}_t(0) e^{-\left( \frac{|x|^2}{4tL}\wedge \frac{m|x|^\beta}{2^\beta}\right) }
	\leq c_4 h(t)^{-d}e^{- \left( \frac{|x|^2}{4 tL}\wedge \frac{m|x|^\beta}{2^\beta} \right) },\quad x\in\Rd.
\end{equation}

We have $\Psi(1/h(t))=1/t$ and it follows from Corollary 10 in \cite{KSz2} with $\gamma=1$ and (\ref{eq:exp}) that
\begin{equation}\label{eq:bar_P_schonwieder}
  \bar{P}_t(B(x,\rho)) \leq c_5 t f\left(|x|/4\right) \rho,
\end{equation}
for $\rho\leq \frac{1}{2}|x|$ and $t>0$,
where
\begin{displaymath}
	f(s) = s^{-1-\alpha}(1+s)^{\kappa} e^{-ms^\beta}, \quad s>0.
\end{displaymath}

We fix $t$ and denote
$$
  g(s)=e^{- \left( \frac{s^2}{4tL}\wedge \frac{m s^\beta}{2^\beta}\right) },\quad s\geq 0.
$$
We note that $g$ is decreasing and continuous on $[0,\infty)$, and the inverse function is given by
\begin{displaymath}
	g^{-1}(s) = \sqrt{4tL\log\frac{1}{s}}\vee \left(\frac{2^\beta \log\frac{1}{s}}{m}\right)^{\frac{1}{\beta}},\quad s\in (0,1].
\end{displaymath}

Using \eqref{eq:tildepest} and (\ref{eq:bar_P_schonwieder}) for $|x|>c_4 \sqrt{t}$, $t>1$ we obtain
\begin{eqnarray*}
  p_t(x) = \tilde{p}_t * \bar{P}_t (x)
  &  =   & \int \tilde{p}_t(x-y)\, \bar{P}_t(dy) \\
  & \leq & \int c_5 [h(t)]^{-d} g(|y-x|) \, \bar{P}_t(dy) \\
  &   =  & c_5 [h(t)]^{-d} \int \int_0^{g(|y-x|) } \, ds\, \bar{P}_t(dy) \\
  &   =  & c_5 [h(t)]^{-d} \int_0^1 \int \indyk{\left\{y\in\Rd:\: g(|y-x|)>s\right\} } \, \bar{P}_t(dy) ds \\
  &   =  & c_5 [h(t)]^{-d} \int_0^1 \bar{P}_t\left(B\left(x,g^{-1}(s) \right)\right) ds \\
  & \leq & c_5 c_6 [h(t)]^{-d} \left( \int_{g(|x|/2)}^1 t f\left(|x|/4\right) 
           g^{-1}(s)\, ds + \int_0^{g(|x|/2)}\, ds \right) \\
  & \leq & c_7 t^{-d/2} \left( t^{3/2} f\left(|x|/4\right) 
           + e^{ -\left(\frac{|x|^2}{16 tL}\wedge \frac{m|x|^\beta}{4^\beta}\right)} \right). \\
  & \leq & c_8 t^{-d/2} \left( e^{ \frac{-|x|^2}{16tL}  } 
                + (1+t^{3/2}|x|^{-1-\alpha}(1+|x|)^\kappa)e^{\frac{-m|x|^\beta}{4^\beta}} \right)\\
  &   =  & c_8 t^{-d/2} \left( e^{ \frac{-|x|^2}{16tL}  } 
                + \left(1+\left( \frac{\sqrt{t}}{|x|}\right)^{3}|x|^{2-\alpha}(1+|x|)\right)^\kappa
								e^{\frac{-m|x|^\beta}{4^\beta}} \right) \\
	& \leq & c_9 t^{-d/2} \left( e^{ \frac{-|x|^2}{16tL}  } 
                + e^{\frac{-m|x|^\beta}{2\cdot 4^\beta}} \right),
\end{eqnarray*}
which yields \eqref{eq:temp1}.
If $|x|\leq c_4 \sqrt{t}$ then \eqref{eq:temp1} follows directly from 
\eqref{eq:p_est_below1}.

Taking $F(s) = c_{10} (s \wedge s^{\alpha-1})$ for $\alpha>1$
and $F(s) = c_{11} s$ for $\alpha\leq 1$ we obtain $F(s)\leq \Psi(s)/s$ for $\alpha>1$ and $s>0$, 
and for $\alpha\leq 1$ and $s\in (0,1)$. From 
Lemma \ref{below_exp2b} we get
\begin{equation}\label{eq:belowpt}
	p_t(x) \geq c_{12} t^{-d/2} e^{-c_{13} |x|^2/t},\quad |x|<c_{14} t,\, t>t_0.
\end{equation}

From Proposition \ref{th:p_est_below} it follows that
\begin{equation}\label{eq:bb}
	p_t(x) 
	\geq  c_{14} t^{1-d/2}  \nu (B(x,c_{15} \sqrt{t})), \quad |x|>\eta \sqrt{t},\, t>t_0,
\end{equation}
and \eqref{eq:temp2} follows from \eqref{eq:belowpt} and \eqref{eq:bb}.

If \eqref{eq:approxtemp} holds and $|x|\geq \eta \sqrt{t}$ then $\nu (B(x,c_{15} \sqrt{t})) \geq c_{16}t^{d/2} e^{- c_{17}|x|^\beta}$ for some
constants $c_{16},c_{17}$. Furthermore, we have
$c_{17} |x|^\beta < c_{18} |x|^2/t $ for $|x|\geq c_{14} t$, and \eqref{densityversion} follows from \eqref{eq:belowpt} and \eqref{eq:bb} for $\eta\sqrt{t} < |x| < c_{14} t$, from 
\eqref{eq:bb} for $|x| \geq c_{14}t$ and from \eqref{eq:p_est_below1} for $|x|\leq \eta \sqrt{t}$.

\end{proof}

\section{High intensity of small jumps}\label{section_high}
We consider now an interesting example, which has been studied in \cite{Mimica1} and \cite{KSz2}. The exact estimates
of transition densities for small $x$ and small $t$ are still unreachable in this case, but using the above results we can improve them
significantly. Let $\nu$ be a L\'evy measure such that

\begin{equation}\label{eq:hi_assum}
	\nu(dx) \asymp |x|^{-d-2} \left[\log\left(\frac{2}{|x|}\right)\right]^{-\beta}\, dx,\quad |x|<1,
\end{equation} 
where $\beta > 1 $. 

This assumption gives following properties of the corresponding semigroup.

\begin{lemat}
  If the L\'evy measure $\nu $ satisfies \eqref{Levy1} and \eqref{eq:hi_assum} then
  \begin{equation}\label{eq:asympPhiPsi}
	  \Phi(\xi) \asymp \Psi(|\xi|) \asymp |\xi|^2 \left[\log(2|\xi|)\right]^{1-\beta},\quad |\xi| \geq 1.
  \end{equation}
  Furthermore, 
  \begin{displaymath}
	  h(t) \asymp t^{1/2} \left[\log\left(\frac{2}{t}\right)\right]^{(1-\beta)/2},\quad t<1,
  \end{displaymath}
  and \eqref{Assum} holds with $t_p=1$.
\end{lemat}
\begin{proof}
  For $|\xi|>1$ by \eqref{eq:PsiH} we have
  \begin{eqnarray*}
    \Phi(\xi) 
     & \leq & 2 |\xi|^2 \int_{|y|\leq 1/|\xi|} |y|^2\, \nu(dy) + 2\int_{|y|>1/|\xi|} \,\nu(dy) \\
     & \leq & c_1 |\xi|^2 \int_0^{\frac{1}{|\xi|}} r^{-1}\left[\log\frac{2}{r}\right]^{-\beta}\, dr
              + c_1 \int_{\frac{1}{|\xi|}}^1 r^{-3} \left[\log\frac{2}{r}\right]^{-\beta}\, dr \\
     &      & + \, 2 \nu(B(0,1)^c) \\
     &  =   & c_1 |\xi|^2 \int_{\log(2|\xi|)}^{\infty} s^{-\beta}\, ds
              + \frac{c_1}{4} \int_{\log 2}^{\log(2|\xi|)} e^{2s} s^{-\beta}\, ds + c_2 \\
     & \leq & c_3 |\xi|^2 \left[\log(2|\xi|)\right]^{1-\beta},
  \end{eqnarray*}
  since $\int_{\log 2}^x  e^{2s}s^{-\beta}\, ds \leq c_4 e^{2x}x^{-\beta+1},$ for $x>\log 2$.
  Similarly, we obtain
  \begin{eqnarray*}
    \Phi(\xi)
     &  =   & \int \left(1-\cos\left(\scalp{\xi}{y}\right)\right)\, \nu(dy) \\
     & \geq & c_5 \int_{|y|\leq 1/|\xi|} |\scalp{\xi}{y}|^2\, \nu(dy)\\
     & \geq & c_6 |\xi|^2 \int_0^{\frac{1}{|\xi|}} r^{-1}\left[\log\frac{2}{r}\right]^{-\beta}\, dr \\
     & =    & c_7 |\xi|^2 \left[\log(2|\xi|)\right]^{1-\beta},\quad |\xi|>1,
  \end{eqnarray*}
  and \eqref{eq:asympPhiPsi} is proved.
  
  For $s>1$, let $g(s)=s^2 \left[\log\left(2s\right)\right]^{1-\beta}$.
  The function $g$ is increasing on $[s_\beta,\infty)$ for some constant $s_\beta\geq 1,$ depending on $\beta$, so there
  exists an increasing inverse function $g^{-1}: [g(s_\beta),\infty))\to [s_\beta,\infty)$. We let
  $\eta(r)=\left( r \left[\log(2r)\right]^{\beta-1}\right)^{1/2}$ for $r>1$. Then there exists $r_\beta$ such that for $r>r_\beta$ we have 
\begin{align*}
    g \left(\eta(r)\right)
    &   = r\, \big(\log(2r)\big)^{\beta-1} \Big[\log\Big( 2 r^{1/2}(\log(2r))^{(\beta-1)/2} \Big)\Big]^{1-\beta}\\
    & =   r\, \big(\log (2r)\big)^{\beta-1} \left[\frac{1}{2} \log(4r) + \frac{1}{2}(\beta-1)\log\log (2r)\right]^{1-\beta}\\
    &= r \left[\frac{\log (4r) +(\beta-1)\log\log (2r)}{2\log (2r)}\right]^{1-\beta}\\
    &\asymp r.
\end{align*}
This shows that $g^{-1}(r)\asymp \eta(r)$ for $r>r_\beta$, and since $h(t)=1/\Psi^{-1}(1/t)\asymp 1/g^{-1}(1/t)$, it
follows that
\begin{displaymath}
  h(t) \asymp t^{1/2} \left[\log\left(\frac{2}{t}\right)\right]^{\frac{1-\beta}{2}},\quad t\in (0,1).
\end{displaymath}
Furthermore, it follows also that there exists constant $c_*>1$ such that
\begin{displaymath}
	g^{-1}(2r) \leq c_* g^{-1}(r),\quad r>g(s_\beta),
\end{displaymath}
and this, for $t<1/g(s_\beta)$, yields

\begin{eqnarray*}
  \int e^{-t\Phi(\xi)}|\xi|\, d\xi
  & \leq & \int_{|\xi|<g^{-1}(1/t)} |\xi|\, d\xi + 
           \int_{|\xi|\geq g^{-1}(1/ t) } e^{-tc_8 g(|\xi|)}|\xi|\,d\xi \\
  & \leq & c_9 \left( \left(g^{-1}(1/t) \right)^{d+1} + \int_{g^{-1}(1/t)}^\infty e^{-c_8 t g(s)} s^d\, ds \right) \\
  & \leq & c_{10} \left( h(t)^{-d-1} + \sum_{k=0}^\infty \int_{g^{-1}(2^k/t)}^{g^{-1}(2^{k+1}/t)} e^{-c_8 t g(s)} s^d\, ds \right) \\
  & \leq & c_{10} \left( h(t)^{-d-1} + 
           \sum_{k=0}^\infty e^{-c_8 2^k} \frac{1}{d+1}\left(g^{-1}\left(\frac{2^{k+1}}{t}\right)\right)^{d+1}\right) \\
  & \leq & c_{10} \left( h(t)^{-d-1} + 
           \frac{1}{d+1} \sum_{k=0}^\infty e^{-c_8 2^k} c_{*}^{(k+1)(d+1)} \left(g^{-1}\left(\frac{1}{t}\right)\right)^{d+1}\right) \\
  & \leq & c_{11} h(t)^{-d-1}.
\end{eqnarray*}

\end{proof}

Except of \eqref{Levy1} we do not assume here anything on the behavior of $\nu$ outside of the ball $B(0,1)$. However it follows easily from \eqref{eq:PsiH} that for every $\nu$ satisfying \eqref{eq:hi_assum} there exists $c_1$ 
such that
$\Psi(s)\geq c_1 s^2$,
for $s<1$, and so the condition \eqref{condition1} holds with $F(s) = c_2 s \left[ \log(e^{\beta-1}+s)\right]^{1-\beta}$. Furthermore, $F^{-1}(s)\asymp s \left[\log (e^{\beta-1}+s)\right]^{\beta-1}$ for $s>0$ 
and from Lemma \ref{below_exp2a} we obtain
\begin{displaymath}
  p_t(x) \geq c_3 h(t)^{-d} e^{-c_4(|x|^2/t) \log(e^{\beta-1}+c_5|x|/t)^{\beta-1} },
\end{displaymath}
for $t>0$, $x\in\Rd$.
If we consider $t<1$ and $|x|<1$, then we get
\begin{displaymath}
  p_t(x) \geq c_3 h(t)^{-d} e^{-c_4(|x|^2/t) \log(e^{\beta-1}+c_5/t)^{\beta-1} } \geq c_{6} h(t)^{-d} e^{-c_{7}(|x|/h(t))^2}.
\end{displaymath}

Combining the estimate with Proposition \ref{th:p_est_below} and \eqref{eq:p_est_below1}
we obtain
\begin{displaymath}
	p_t(x) \geq   c_{8} \min \left\{ t^{-d/2} \left(\log \frac{2}{t} \right)^{d(\beta-1)/2},\frac{t
                        }{|x|^{d+2} \left(\log\frac{2}{|x|}\right)^{\beta}} +\, h(t)^{-d}e^{ -c_{7}(|x|/h(t))^2} \right\},
\end{displaymath}
for $|x|<1$ and $t<1$.

If the L\'evy measure $\nu$ has a density which is bounded on $B(0,1)^c$ then from Theorem 1 in \cite{KSz2} we obtain
the following upper estimate.
\begin{displaymath}
	p_t(x) \leq   c_{8} \min \left\{ t^{-d/2} \left(\log \frac{2}{t} \right)^{d(\beta-1)/2},\frac{t
                        }{|x|^{d+2} \left(\log\frac{2}{|x|}\right)^{\beta}} +\, h(t)^{-d}e^{ \frac{-c_{9}|x|}{h(t)}
                  \log\left(1+\frac{c_{10}|x|}{h(t)}\right)}\right\},
\end{displaymath}
for $|x|<1$ and $t<1$.

We see that we do not have sharp both sides estimates in this case, 
however the new results of \cite{Mimica_new}
obtained for subordinated Brownian motion (contained in the case of $\beta=2$ here)
show that the lower estimate is optimal (note that $h(t)\leq |x|$ under the assumption
$t\Psi(|x|^{-1})\leq 1$ given in \cite{Mimica_new}).

\section{Appendix}

\begin{proof}[Proof of Proposition \ref{th:p_est_below}]
Similarly as in the proof of Theorem \ref{t:esttemp1} we consider
$\tilde{p}_t^r$ and $\bar{P}^r_t$, noting that \eqref{eq:FTildePEstimate} also holds and hence the densities
$\tilde{p}_t^r$ exist for every $t\in (0,t_p)$ and $r>0$.
First we will prove that
there exist constants $c_1=c_1(d)$, $c_2=c_2(d,M_0)$, $c_3=c_3(d)$ such that for every $a\in (0,1]$ we have
\begin{equation}\label{eq:TildaBelow}
  \tilde{p}^{h(at)}_t(y) \geq c_1 \left(h(t)\right)^{-d},  
\end{equation}
provided $|y|\leq c_2 e^{-c_3/a} h(t)$, $t\in (0,t_p)$.  

By symmetry of $\nu$ we have
$$
  \Fourier(\tilde{p}^{h(at)}_t)(\xi)
   \geq  |\Fourier(p_t)(\xi)|,\quad \xi\in\Rd,\, t\in (0,t_p),
$$ and this and Lemma 4 in \cite{KSz2} yield
\begin{eqnarray*}
  \tilde{p}^{h(at)}_t(0)
         & \geq & (2\pi)^{-d} \int e^{-t\Re(\Phi(\xi))}\, d\xi \\
         & \geq & c_4 \left(h(t)\right)^{-d}, \quad t\in (0,t_p),
\end{eqnarray*}
where the constant $c_4$ depends only on $d$.
It follows from (\ref{eq:FTildePEstimate}) that
$$
  |\Fourier(\tilde{p}^{h(at)}_t)(\xi)| \leq |\Fourier(p_t)(\xi)|e^{2t\nu(B(0,h(at))^c}
$$
and since by \eqref{eq:PsiH} we have $\nu(B(0,r)^c)\leq (1/L_0)\Psi(1/r)$, we obtain
\begin{displaymath}
	|\Fourier(\tilde{p}^{h(at)}_t)(\xi)| \leq e^{-t\Re(\Phi(\xi))} e^{2t\Psi(1/h(at))/L_0} = e^{-t\Re(\Phi(\xi))} e^{2/(aL_0)},
\end{displaymath}

and for every $j\in\{1,\dots,d\}$ by (\ref{Assum}) we get	
\begin{eqnarray*}
  \left|\frac{\partial \tilde{p}^{h(at)}_t }{\partial y_j}(y)\right|
  &   =  & \left|(2\pi)^{-d}\int
           (-i)\xi_j e^{-i\scalp{y}{\xi}}
           \Fourier(\tilde{p}^{h(at)}_t)(\xi)d\xi\right| \\
  & \leq & (2\pi)^{-d} e^{2/aL_0} \int_{\Rd} e^{-t\Re(\Phi(\xi))}|\xi|\, d\xi \\
    & \leq & c_6 e^{2/aL_0} \left(h(t)\right)^{-d-1},
\end{eqnarray*}
with $c_6=c_6(d,M_0)$.
It follows that
$$
  \tilde{p}^{h(at)}_t(y)\geq c_4\left(h(t)\right)^{-d}-dc_6 e^{2/aL_0} \left(h(t)\right)^{-d-1} |y|
  \geq \frac{1}{2}c_4 \left(h(t)\right)^{-d},
$$
provided  $|y|\leq \frac{c_4}{2dc_6}e^{-2/aL_0} h(t)$,
which clearly yields (\ref{eq:TildaBelow}).
  
Let $a\in(0,1)$ and $t\in (0,t_p)$.
For $r>0$, $|x|>r+h(at)$ by (\ref{eq:exp}) and \eqref{eq:PsiH} we get
$$
  \bar{P}^{h(at)}_t(B(x,r)) \geq e^{-1/aL_0} t \bar{\nu}_{h(at)}(B(x,r)) = e^{-1/aL_0} t \nu (B(x,r)).
$$
This and (\ref{eq:TildaBelow}) for $x\in (0,t_p)$ yield 
\begin{eqnarray*}
  p_t(x) 
  &  =    & \tilde{p}^{h(at)}_t * \bar{P}_t^{h(at)}(x) \\
  &  =   & \int \tilde{p}_t^{h(at)}(x-z)\bar{P}_t^{h(at)}(dz) \\
  & \geq & c_1 \int_{|z-x|<c_2e^{-c_3/a }h(t)} 
             \left(h(t)\right)^{-d} \bar{P}_t^{h(at)}(dz) \\
  &  =   & c_1 \left(h(t)\right)^{-d} \bar{P}_t^{h(at)}(B(x,c_2e^{-c_3/a}h(t))) \\
  & \geq & c_1 e^{-1/aL_0} t\left(h(t)\right)^{-d} \nu(B(x,c_2e^{-c_3/a}h(t))), 
\end{eqnarray*}
provided $|x|>h(at)+c_2e^{-c_3/a}h(t)$.
Using the fact that $h(at)/h(t)\leq \sqrt{2a/L_0}$ for $a<L_0/2$ and $t>0$ (see the proof of 
Lemma 11 in \cite{KSz2}) we choose $a\in(0,1)$ such that
$h(at)/h(t)+c_2e^{-c_3/a}\leq \eta$ and we obtain (\ref{eq:p_est_below0}).
\end{proof}


\begin{thebibliography}{2}

\bibitem{BG60} R. M. Blumenthal and R. K. Getoor, {\it Some theorems on stable processes},
Trans. Amer. Math. Soc.  95 (1960), 263--273.

\bibitem{BGR13} K. Bogdan, T. Grzywny, M. Ryznar, {\it Density and tails of unimodal convolution semigroups}, 
  J. Funct. Anal 266, No. 6, 3543–-3571 (2014).
\bibitem{BJ07} K. Bogdan, T. Jakubowski, {\it Estimates of heat kernel of 
fractional Laplacian perturbed by gradient operators}, Comm. Math. Phys. 271 (1) 2007, 179--198.
\bibitem{BS2007} K. Bogdan, P. Sztonyk, {\it Estimates of potential kernel and
Harnack's  inequality for anisotropic fractional Laplacian,} Stud. Math. 181, No. 2, 101-123 (2007).
\bibitem{CKS} E.A. Carlen,  S. Kusuoka, D.W. Stroock, {\it Upper bounds for symmetric Markov transition functions,}
Ann. Inst. H. Poincaré, Probab. Stat. Suppl. 23, 245–287 (1987).
\bibitem{ChKimKum1} Z.-Q. Chen, P. Kim, T Kumagai,
{\it Global Heat Kernel Estimates for Symmetric Jump Processes},
Trans. Amer. Math. Soc. 363, no. 9, 5021--5055 (2011).
\bibitem{ChKum08} Z.-Q. Chen, T. Kumagai, {\it Heat kernel estimates for jump processes of mixed types on metric measure spaces},
  Probab. Theory Relat. Fields 140, No. 1-2, 277-317 (2008).
\bibitem{ChKimKum2} Z.-Q. Chen, P. Kim, T. Kumagai, Weighted Poincar\'e inequality and heat
kernel estimates for finite range jump processes, Math. Ann. 342 (2008),
No. 4, 833-883.
\bibitem{D91} J. Dziuba\'nski, {\it Asymptotic behaviour of
 densities of stable semigroups of measures,} Probab. Theory Related
 Fields 87 (1991), 459-467.
\bibitem{G93} P. G{\l}owacki, {\it Lipschitz continuity of densities of stable semigroups of measures,}
Colloq. Math. 66, No.1, 29-47 (1993).
\bibitem{GH93} P. G{\l}owacki, W. Hebisch, {\it Pointwise estimates for
    densities of stable semigroups of measures,} Studia Math. 104
  (1993), 243-258.
\bibitem{Grz2013} T. Grzywny, {\it On Harnack inequality and H\"older regularity for isotropic
unimodal L\'evy processes},  Potential Anal. 41, No. 1, 1–29 (2014).
\bibitem{H94} S. Hiraba, {\it Asymptotic behaviour of densities of multi-dimensional stable distributions,} Tsukuba J. Math. 18, No.1, 223--246 (1994).
\bibitem{H03} S. Hiraba, {\it Asymptotic estimates for densities of
multi-dimensional stable distributions},  Tsukuba J. Math.  27
(2003),  no. 2, 261--287.
\bibitem{JKLSch2012} N. Jacob, V. Knopova, S. Landwehr, R. Schilling, {\it A geometric interpretation of the transition density of a
             symmetric {L}\'evy process,} Sci. China Math. 55 (2012), no. 6, 1099--1126.
  \bibitem{KSz1} K. Kaleta, P. Sztonyk, {\it Upper estimates of transition densities for stable-dominated semigroups,}  
    J. Evol. Equ. 13, No. 3, 633–-650 (2013).
  \bibitem{KSz2} K. Kaleta, P. Sztonyk, {\it Estimates of transition densities and their derivatives for jump L\'evy processes,} preprint.
  \bibitem{KSz3} K. Kaleta, P. Sztonyk, {\it Small time sharp bounds for kernels of convolution semigroups,} to appear in
    J. Anal. Math.
  \bibitem{Knop13} V. Knopova, {\it Compound kernel estimates for the transition probability density of a L\'evy process in $\R^n$},  
    Teor. Imovir. Mat. Stat. No. 89, 51--63 (2013); translation in Theory Probab. Math. Statist. No. 89, 57–-70 (2014).
\bibitem{KnopKul} V. Knopova, A. Kulik, {\it Exact asymptotic for distribution densities of L\'evy functionals.} 
  Electronic Journal of Probability 16, 1394-1433 (2011).
\bibitem{KSch2012} V. Knopova, R. Schilling, {\it Transition density estimates for a class of L\'evy and L\'evy-type processes.} J. Theoret. Probab. 25 (1), 144-170 (2012).
\bibitem{KR13} T. Kulczycki, M. Ryznar,  {\it Gradient estimates of harmonic functions and transition densities for L\'evy processes}, 
  Trans. Amer. Math. Soc., to appear, arXiv:1307.7158.
\bibitem{Mimica1} A. Mimica, {\it Heat kernel upper estimates for symmetric jump processes with small jumps of high intensity.} Potential Anal. 36 (2012), no. 2, 203–-222. 
\bibitem{Mimica_new} A. Mimica, {\it Heat kernel estimates for subordinate Brownian motions}, preprint.
\bibitem{Picard97} J. Picard, {\it Density in small time at accessible points for jump processes,}
Stochastic Process. Appl. 67 (1997), no. 2, 251–-279.
\bibitem{PT69} W.E. Pruitt, S.J. Taylor, {\it The potential kernel and hitting probabilities for the general stable process in $R\sp N$,} Trans. Am. Math. Soc. 146 (1969), 299-321.  
\bibitem{SchSW12} R. Schilling, P. Sztonyk and J. Wang, {\it Coupling property and gradient estimates for L\'evy processes via the symbol.}
Bernoulli 18 (2012), 1128-1149.
\bibitem{S10a} 
  P. Sztonyk, {\it  Regularity of harmonic functions for anisotropic fractional Laplacians.}
   Math. Nachr. 283 (2010) No. 2, 289-311.
\bibitem{S10} 
  P. Sztonyk, {\it Estimates of tempered stable densities,} J. Theoret. Probab. 23 (2010) No. 1, 127-147 .
\bibitem{S11}
  P. Sztonyk, {\it Transition density estimates for jump L\'evy processes,} Stochastic Process. Appl. 121 (2011), 1245-1265.
\bibitem{W07} T. Watanabe, {\it Asymptotic estimates of
multi-dimensional stable densities and their applications},
Trans. Am. Math. Soc. 359, No. 6, 2851-2879 (2007).


\end{thebibliography}
\end{document}